\documentclass[11pt,reqno]{amsart}
\usepackage{mathrsfs}
\usepackage{url}
\usepackage{mathtools}
\usepackage{latexsym,epsfig,amssymb,amsmath,amsthm,color,url,bm}
\usepackage[inline,shortlabels]{enumitem}
\usepackage{hyperref}
\usepackage[foot]{amsaddr}
\usepackage{amsmath,amsbsy}
\usepackage{mwe}
\RequirePackage[numbers]{natbib}
\usepackage{mathptmx}
\usepackage[text={16cm,24cm}]{geometry}

\allowdisplaybreaks 
 \numberwithin{equation}{section}
\newtheorem{theorem}{Theorem}[section]

\newtheorem{lemma}[theorem]{Lemma}

\newcommand\Item[1][]{%
  \ifx\relax#1\relax  \item \else \item[#1] \fi
  \abovedisplayskip=0pt\abovedisplayshortskip=0pt~\vspace*{-\baselineskip}}

\theoremstyle{definition}
\newtheorem{defn}[theorem]{Definition}

\theoremstyle{definition}
\newtheorem{remark}[theorem]{Remark}

\DeclareMathOperator{\Prob}{\mathbf{P}}
\DeclareMathOperator{\E}{\mathbf{E}}

\DeclareMathOperator{\tv}{TV}

\title[]{Phase transition in percolation games on rooted Galton-Watson trees}
\date{}
\author{Sayar Karmakar, Moumanti Podder, Souvik Roy, Soumyarup Sadhukhan}
\address{Sayar Karmakar, University of Florida, 230 Newell Drive, Gainesville, Florida 32605, USA.}
\address{Moumanti Podder, Indian Institute of Science Education and Research (IISER) Pune, Dr.\ Homi Bhabha Road, Pashan, Pune 411008, Maharashtra, India.}
\address{Souvik Roy, Indian Statistical Institute, 203 Barrackpore Trunk Road, Kolkata 700108, West Bengal, India.}
\address{Soumyarup Sadhukhan, Indian Institute of Technology, Kalyanpur, Kanpur, Uttar Pradesh 208016, India.}
\email{sayarkarmakar@ufl.edu}
\email{moumanti@iiserpune.ac.in}
\email{souvik.2004@gmail.com}
\email{soumyarup.sadhukhan@gmail.com}

\begin{document}
\bibliographystyle{plainnat}
\begin{abstract}
We study the \emph{bond percolation game} and the \emph{site percolation game} on the rooted Galton-Watson tree $T_{\chi}$ with offspring distribution $\chi$. We obtain the probabilities of win, loss and draw for each player in terms of the fixed points of functions that involve the probability generating function $G$ of $\chi$, and the parameters $p$ and $q$. Here, $p$ is the probability with which each edge (respectively vertex) of $T_{\chi}$ is labeled a trap in the bond (respectively site) percolation game, and $q$ is the probability with which each edge (respectively vertex) of $T_{\chi}$ is labeled a target in the bond (respectively site) percolation game. We obtain a necessary and sufficient condition for the probability of draw to be $0$ in each game, and we examine how this condition simplifies to yield very precise phase transition results when $\chi$ is Binomial$(d,\pi)$, Poisson$(\lambda)$, or Negative Binomial$(r,\pi)$, or when $\chi$ is supported on $\{0,d\}$ for some $d \in \mathbb{N}$, $d \geqslant 2$. It is fascinating to note that, while all other specific classes of offspring distributions we consider in this paper exhibit phase transition phenomena as the parameter-pair $(p,q)$ varies, the probability that the bond percolation game results in a draw remains $0$ for \emph{all} values of $(p,q)$ when $\chi$ is Geometric$(\pi)$, for all $0 < \pi \leqslant 1$. By establishing a connection between these games and certain \emph{finite state probabilistic tree automata} on rooted $d$-regular trees, we obtain a precise description of the regime (in terms of $p$, $q$ and $d$) in which these automata exhibit \emph{ergodicity} or \emph{weak spatial mixing}.
\end{abstract}

\subjclass[2020]{05C57, 37B15, 37A25, 68Q80}

\keywords{percolation games on rooted random trees; two-player combinatorial games; rooted Galton-Watson trees; probabilistic tree automata; ergodicity; probability of draw; weak spatial mixing; phase transition; rooted regular trees; bond percolation; site percolation}

\maketitle

\section{Introduction}\label{sec:intro} 
\subsection{Overview of the paper}\label{subsec:overview} The work accomplished in this paper is the result of our curiosity regarding \emph{percolation games} on rooted trees, both deterministic and random. Whereas \emph{percolation} itself has been thoroughly researched for decades by both mathematicians and physicists (see, for instance, \cite{bollobas2006percolation}, \cite{grimmett1999percolation} and \cite{stauffer2018introduction}), the notion of percolation games -- more specifically, \emph{node percolation games} or \emph{site percolation games} -- to the best of our knowledge, was introduced in \cite{holroyd2019percolation}. In this paper, we address both site percolation games and \emph{bond percolation games}. The latter is analogous, in some sense, to the former, but with the edges of the graph under consideration replacing the vertices in their roles in the game (see \S\ref{subsec:games} for a detailed description of each game). 

In this paper, each of these games is played, and the probabilities of their various outcomes analyzed, on a \emph{rooted Galton-Watson} (henceforth abbreviated as GW) tree $T_{\chi}$, with root $\phi$ and offspring distribution $\chi$. We impose no assumption on $\chi$ whatsoever, except that the probability of the event that $\phi$ has no child is strictly less than $1$. We assign, to each edge (respectively, each vertex) of $T_{\chi}$, independently, a label that reads \emph{trap} with probability $p$, \emph{target} with probability $q$, and \emph{safe} with probability $1-p-q$, where $0 < p+q < 1$, following which two players take turns to move a token along the edges of $T_{\chi}$ (always moving \emph{away} from the root). A player wins the bond percolation game (respectively, the site percolation game) if she is able to move the token \emph{along} an edge (respectively, \emph{to} a vertex) that has been labeleed a target, or force her opponent to move the token along an edge (respectively, to a vertex) that has been labeled a trap. We deduce a necessary and sufficient condition for the probability of draw to be $0$ in each of these games when they are played on $T_{\chi}$ for \emph{any} $\chi$, following which we consider specific classes of probability distributions in place of $\chi$. This allows us to gain further insight into the \emph{regimes}, in terms of the underlying parameters, in which neither of the games results in a draw with positive probability.    

The two primary sources of inspiration for this paper are \cite{holroyd2019percolation} and \cite{holroyd2021galton}. In the former, each site $(x,y) \in \mathbb{Z}^{2}$ of the $2$-dimensional square lattice is assigned, independently, a label that reads \emph{trap} with probability $p$, \emph{target} with probability $q$ and \emph{open} with probability $1-p-q$, where $p$ and $q$ are pre-fixed parameters such that $0 \leqslant \min\{p, q\} < p+q \leqslant 1$. Two players take turns to make moves, where a move involves relocating a token from its current position, say the site $(x,y)$, to one of the sites $(x+1,y)$ and $(x,y+1)$. A player wins if she either succeeds in moving the token to a site that has been labeled a target, or forces her opponent to move the token to a site that has been labeled a trap. The game continues for as long as the token stays on open sites, and this may happen indefinitely, thereby leading to a draw. In \cite{holroyd2019percolation}, it is shown, exploiting the connection between this site percolation game and a suitably defined \emph{probabilistic cellular automaton} (henceforth abbreviated as PCA), and using the technique of \emph{weight functions} or \emph{potential functions}, that the probability of the event that this game results in a draw is $0$ (and equivalently, the PCA is \emph{ergodic}) whenever $p+q > 0$.

In \cite{holroyd2021galton}, \emph{normal}, \emph{mis\`{e}re} and \emph{escape} games are considered on the rooted GW tree $T_{\chi}$. In each of these games, two players take turns to slide a token along the edges of $T_{\chi}$, always moving \emph{from} the parent vertex \emph{to} the child vertex. A player loses the normal game if she is unable to make a move (which happens when the token is stagnated at a leaf vertex of $T_{\chi}$). A player loses the mis\`{e}re game if her opponent is unable to make a move. In the escape game, the player titled \emph{Stopper} wins if \emph{either} of the two players is unable to make a move, else her opponent, titled \emph{Escaper}, wins. The probabilities of win and loss (and draw, in case of the normal and mis\`{e}re games) for each player are derived in terms of the fixed points of certain functions that involve the \emph{probability generating function} $G$ of the offspring distribution $\chi$.

\subsection{A formal introduction to the games}\label{subsec:games} In this paper, we study the site percolation game (inspired by the version studied in \cite{holroyd2019percolation}, as described above) and the bond percolation game on the rooted GW tree $T_{\chi}$ with offspring distribution $\chi$. We assume that $\chi$ is supported on $\mathbb{N}_{0}$ (which ensures that $\chi$ takes only finite, non-negative integer values), and that $0 \leqslant \chi(0) < 1$ (here, $\chi(m)$ denotes the probability that a random variable, that follows the distribution $\chi$, takes the value $m$, for any $m \in \mathbb{N}_{0}$). To begin with, we recall for the reader that $T_{\chi}$ is generated as follows. The root $\phi$ is allowed to have a random number $X$ of children according to the law $\chi$. If $X$ takes the value $0$, the process stops and $T_{\chi}$ comprises only the root $\phi$. Otherwise, conditioned on $X=m$ for any $m \in \mathbb{N}$ that belongs to the support of $\chi$, we name the children of the root $u_{1}, u_{2}, \ldots, u_{m}$. We let $u_{i}$ have $X_{i}$ children for $1 \leqslant i \leqslant m$, where $X_{1}, X_{2}, \ldots, X_{m}$ are independent and identically distributed (i.i.d.)\ $\chi$. We continue thus, and the rooted random tree thus generated is denoted by $T_{\chi}$. It has a positive probability of surviving, i.e.\ of being infinite, if and only if the expectation of $\chi$ is strictly greater than $1$. For an overview of the GW branching process and its many applications, we refer the reader to \cite{athreya2004branching}.  

We visualize $T_{\chi}$ as a directed graph, as follows: given any edge $\{u,v\}$ of $T_{\chi}$, connecting vertices $u$ and $v$ where $u$ is the parent of $v$, we let $(u,v)$ denote the directed edge that travels \emph{from} $u$ \emph{towards} $v$. Next, we choose and fix a pair $(p,q)$ of parameters from the set 
\begin{equation}
I = \left\{(p,q) \in [0,1]^{2}: 0 < p+q < 1\right\}.\label{I_defn}
\end{equation}
Each game involves a token, and two players that we henceforth refer to as $A$ and $B$. In each game, $A$ and $B$ take turns to make moves, where a move involves relocating the token along a directed edge of $T_{\chi}$ (i.e.\ if the token is currently situated at some vertex $u$ of $T_{\chi}$, the player whose turn it is to move next is permitted to move the token to any child of $u$, if such a child exists).
\begin{enumerate}
\item In case of the bond percolation game, each edge of $T_{\chi}$ is assigned, independently, a label that reads \emph{trap} with probability $p$, \emph{target} with probability $q$, and \emph{safe} with the remaining probability $1-p-q$. A player wins this game if she is able to move the token along a directed edge that has been labeled a target, or force her opponent to move the token along a directed edge that has been labeled a trap, or confine her opponent to a leaf vertex (from where her opponent cannot move the token any further).
\item In case of the site percolation game, each vertex of $T_{\chi}$ is assigned, independently, a label that reads trap with probability $p$, target with probability $q$, and safe with the remaining probability $1-p-q$. A player wins this game if she is able to move the token to a vertex that has been labeled a target, or force her opponent to move the token to a vertex that has been labeled a trap or that is a leaf vertex that has been labeled safe.
\end{enumerate}
Once a realization of $T_{\chi}$, followed by a realization of the random assignment of trap / target / safe labels (be it to the edges of this realization of $T_{\chi}$ or to the vertices), has been generated, it is revealed in its entirety to both $A$ and $B$ -- this is what makes each of these games a \emph{perfect information} game. A game that continues for eternity is said to result in a draw. In case of the bond percolation game, this happens if the token traverses \emph{only} along edges that have been labeled safe, whereas in case of the site percolation game, this happens if the token \emph{only ever} visits vertices that have been labeled safe. We assume both players to play \emph{optimally}, i.e.\ if the game is destined to end in a finite number of rounds, then the player who wins tries to wrap up the game as quickly as possible, whereas the player who loses tries to prolong its duration as much as possible.

\subsection{The main results of this paper}\label{subsec:main_results} For each of the two games described above, we investigate the probabilities of the various possible outcomes, i.e.\ win for player $A$, win for player $B$, and draw for both the players. Of particular interest is the phenomenon of \emph{phase transition} pertaining to the probability of draw: in certain \emph{regimes}, i.e.\ for certain values of the parameter-pair $(p,q)$ (that depend on the offspring distribution $\chi$ under consideration), the probability of the event that the game results in a draw is $0$, while for other regimes, this is strictly positive. We find a necessary and sufficient condition for the probability of draw to be $0$, then examine what this condition reduces to when we consider specific classes of probability distributions in place of the offspring distribution $\chi$.

In what follows, we let $G$ denote the probability generating function of $\chi$, i.e.\ $G(x) = \sum_{m=0}^{\infty}x^{m}\chi(m)$ for all $x \in [0,1]$. We define the function $g_{p,q}: [0,1] \rightarrow [0,1]$ as
\begin{align}
g_{p,q}(x) = (1-q) - (1-p-q)G(x), \text{ for } x \in [0,1].\label{GW_g_defn}
\end{align}
Henceforth, given any function $f: S \rightarrow S$, where $S$ is a subset of the real line $\mathbb{R}$, we let $f^{(n)}$ denote the $n$-fold composition of $f$ with itself, for each $n \in \mathbb{N}$, $n \geqslant 2$. In what follows, we consider the fixed points of the function $g_{p,q}^{(2)}$ in the interval $[0,1]$. We note that, since $0 \leqslant g_{p,q}^{(2)}(x) \leqslant 1$ for each $x \in [0,1]$, hence $g_{p,q}^{(2)}(x)-x$ is non-negative at $x=0$ and non-positive at $x=1$, guaranteeing the existence of at least one fixed point in $[0,1]$ by the intermediate value theorem (since $g_{p,q}^{(2)}$ is a continuous function). Yet another way to argue that at least one fixed point of $g_{p,q}^{(2)}$ exists in $[0,1]$ is to employ Brouwer's fixed-point theorem, which states that every continuous function from a convex compact subset of a Euclidean space to itself has a fixed point. 

\begin{theorem}\label{thm:main_GW_bond}
Let $w = w(p,q)$ denote the probability of the event that the bond percolation game starting with the token at the root $\phi$ of $T_{\chi}$ is won by the player who plays the first round, and let $\ell = \ell(p,q)$ denote the probability that this game is lost by the player who plays the first round. Then 
\begin{enumerate}
\item \label{main_GW_bond_1} $w' = p + (1-p-q)w$ equals the minimum positive fixed point of the function $g_{p,q}^{(2)}$, where $g_{p,q}$ is as defined in \eqref{GW_g_defn}, 
\item \label{main_GW_bond_2} $\ell' = (1-p-q)\ell$ equals $(1-q) - g_{p,q}(w')$, and
\item \label{main_GW_bond_3} and the probability that this game results in a draw is $0$ if and only if the function $g_{p,q}^{(2)}$ has a \emph{unique} fixed point in $[0,1]$.
\end{enumerate}
\end{theorem}  

\begin{theorem}\label{thm:main_GW_site}
Let $\widetilde{w} = \widetilde{w}(p,q)$ denote the probability of the event that the site percolation game starting with the token at the root $\phi$ of $T_{\chi}$ is won by the player who plays the first round, and $\widetilde{\ell} = \widetilde{\ell}(p,q)$ the probability that this game is lost by the player who plays the first round. Then 
\begin{enumerate}
\item \label{main_GW_site_1} $\widetilde{w}$ equals the minimum positive fixed point of the function $g_{p,q}^{(2)}$, where $g_{p,q}$ is as defined in \eqref{GW_g_defn}, 
\item \label{main_GW_site_2} $\widetilde{\ell} = q + (1-p-q)G(\widetilde{w})$, and
\item \label{main_GW_site_3} the probability that this game results in a draw is $0$ if and only if the function $g_{p,q}^{(2)}$ has a \emph{unique} fixed point in $[0,1]$.
\end{enumerate} 
\end{theorem}

It is evident, from \eqref{main_GW_bond_3} of Theorem~\ref{thm:main_GW_bond} and \eqref{main_GW_site_3} of Theorem~\ref{thm:main_GW_site}, that given \emph{any} offspring distribution $\chi$, the values of the parameter-pair $(p,q)$ for which the probability of draw in the bond percolation game is $0$ are \emph{precisely} the values of $(p,q)$ for which the probability of draw in the site percolation game is also $0$. Therefore, in the theorems that follow, we only describe the regimes (i.e.\ the sets of values of $(p,q)$, governed by the underlying parameter(s) of the specific classes of offspring distributions that we consider) in which the probability of draw in the bond percolation game is $0$. 

\begin{theorem}\label{thm:binomial_bond}
When played on $T_{\chi}$ where $\chi$ is the Binomial$(d,\pi)$ distribution, for any $0 < \pi \leqslant 1$ and any $d \in \mathbb{N}$ with $d \geqslant 2$, the probability that the bond percolation game results in a draw is $0$ if and only if $(1-p-q) \pi (1-\pi q)^{d-1} \leqslant (d+1)^{d-1}d^{-d}$.
\end{theorem}
 
\begin{theorem}\label{thm:Poisson_bond}
When played on $T_{\chi}$ where $\chi$ is Poisson$(\lambda)$, for any $\lambda > 0$, the probability that the bond percolation game results in a draw is $0$ if and only if $(1-p-q) \lambda e^{-\lambda q} \leqslant e$.
\end{theorem}

\begin{theorem}\label{thm:neg_bin_bond}
When $\chi$ is Negative Binomial$(r,\pi)$, for $r \in \mathbb{N}$ and $0 < \pi \leqslant 1$, the probability of draw in the bond percolation game played on $T_{\chi}$ is $0$ if and only if $(r-1)^{r+1}(1-\pi)(1-p-q)\pi^{r} \leqslant (q+\pi-q\pi)^{r+1}r^{r}$. 
\end{theorem}

\begin{theorem}\label{thm:0_d_bond}
When $\chi$ is the probability distribution with $\chi(0) = 1-\pi$ and $\chi(d) = \pi$ for some $0 < \pi < 1$ and some $d \in \mathbb{N}$ with $d \geqslant 2$, the probability of draw in the bond percolation game played on $T_{\chi}$ is $0$ if and only if $(1-p-q) \pi \left\{\pi(1-q) + p(1-\pi)\right\}^{d-1} \leqslant \frac{(d+1)^{d-1}}{d^{d}}$.
\end{theorem}

Our final main result deviates a little in flavour from the theorems stated above, and concerns itself with the average or expected duration of the bond percolation game on $T_{\chi}$, for any offspring distribution $\chi$.
\begin{theorem}\label{thm:avg_dur}
Suppose the probability of the event that the bond percolation game, played on $T_{\chi}$, results in a draw is $0$. As long as the derivative of the function $g_{p,q}(x)$ as well as the derivative of $g_{p,q}^{(2)}(x)$ at $x=w'$ (with $w'$ as defined in Theorem~\ref{thm:main_GW_bond}) is strictly less than $1$, the expectation of the (random) time for which the game continues is finite.
\end{theorem}

\subsection{Motivations and brief review of the literature}\label{subsec:lit_review}
This work bears intimate connections with a diverse array of topics spanning probability, combinatorics, theoretical computer science and physics. This is either because of the objects we study in this paper, or because of the approach we adopt to analyze said objects. Of particular interest are the connections evident with \emph{recursive distributional equations}, \emph{automata theory} (more specifically, with \emph{finite state probabilistic tree automata}, for which we refer the reader to \S\ref{sec:ergodicity}), \emph{spatial mixing} phenomena in models of \emph{statistical mechanics} on rooted trees, and of course, the classically studied topic of \emph{percolation} in physics.

We begin \S\ref{subsec:lit_review} with some discussion on percolation games, along with a few other instances of \emph{two-player combinatorial games}, and the ties that such games often exhibit, via the recurrence relations arising out of these games, with recursive distributional equations, as well as with cellular automata or finite-state tree automata (whether deterministic or probabilistic). Recall from \S\ref{subsec:overview} that \cite{holroyd2019percolation} and \cite{holroyd2021galton} serve as the two main sources of motivation for this work. We have already included in \S\ref{subsec:overview} a brief description of the site percolation game on the $2$-dimensional square lattice that has been studied in \cite{holroyd2019percolation}. In \cite{bhasin2022class}, the following modification of this game is considered: a move now involves relocating the token from where it is currently situated, say the vertex $(x,y)$ of $\mathbb{Z}^{2}$, to any of the vertices $(x+2,y)$, $(x+1,y+1)$ and $(x,y+2)$. In both \cite{holroyd2019percolation} and \cite{bhasin2022class}, it is shown that the probability of the event that the site percolation game under consideration results in a draw is $0$ if and only if a suitably defined (via the recurrence relations arising from the game) $1$-dimensional probabilistic cellular automaton (PCA) is ergodic, following which it is shown, using the technique of weight functions, that this PCA is ergodic whenever $p+q > 0$. In \cite{bhasin2022ergodicity}, yet another modification of the site percolation game on $\mathbb{Z}^{2}$ is considered, with the token now being allowed to be moved from $(x,y)$ to one of $(x,y+1)$ and $(x+1,y+1)$ when $x$ is even, and from $(x,y)$ to one of $(x+1,y+1)$ and $(x+2,y+1)$ when $x$ is odd. The probability of draw in this game, for various values of the underlying parameters $p$ and $q$, is explored alongside ergodicity properties of a $1$-dimensional \emph{generalized} PCA whose stochastic update rule captures the recurrence relations arising out of this game. We mention here that \cite{bresler2022linear} employs the technique of weight functions to come up with computer-assisted proofs of ergodicity for two classes of PCAs.  

Recall the brief discussion on the contents of \cite{holroyd2021galton} included in \S\ref{subsec:overview}. In \cite{podder2022combinatorial}, a generalization of each of the normal and mis\`{e}re games studied in \cite{holroyd2021galton} is investigated: here, the token is allowed to be moved from where it is currently located, say a vertex $u$ of the rooted GW tree, to any descendant $v$ of $u$ such that $v$ is at a distance at most $k$ away from $u$, for some pre-specified $k \in \mathbb{N}$. In both \cite{holroyd2021galton} and \cite{podder2022combinatorial}, the key component of the analysis constitutes the recurrence relations that arise out of the games under consideration, and the approach is similar in flavour to that adopted in this paper. It is worthwhile to note here that our work in this paper can be seen as a generalization of the set-up considered in \cite{holroyd2021galton}. 

In \cite{johnson2020random}, given a deterministic finite state tree automaton $A$, with alphabet $\mathcal{S}$ and an update rule $f_{A}: \mathbb{N}_{0}^{\mathcal{S}} \rightarrow \mathcal{S}$ (see \S\ref{subsec:tree_automata} for relevant definitions), a probability distribution $\nu$, supported on $\mathcal{S}$, is called a \emph{fixed point} of $A$ if the following is true: if the children of the root $\phi$ of a GW tree $T_{\chi}$ are assigned i.i.d.\ $\nu$ states from $\mathcal{S}$, the (random) state induced at $\phi$ via $f_{A}$ will again have law $\nu$. A map $\iota$ from the space of all locally finite rooted trees to $\mathcal{S}$ is called an \emph{interpretation} for $A$ if the following is true: given any rooted tree $T$, and any vertex $u$ of $T$ with children $u_{1}, \ldots, u_{r}$, the state $\iota(T(u))$ is obtained by applying the update rule $f_{A}$ to the tuple $(n_{\sigma}: \sigma \in \mathcal{S})$, where $n_{\sigma}$ denotes the number of $1 \leqslant i \leqslant r$ such that $\iota(T(u_{i})) = \sigma$. Here, $T(u)$ (respectively $T(u_{i})$, for $1 \leqslant i \leqslant r$) denotes the sub-tree of $T$ induced on the subset of vertices comprising $u$ (respectively $u_{i}$) and its descendants. A fixed point $\nu$ of $A$ is said to be \emph{interpretable} if there exists an interpretation $\iota$ for $A$ such that the law of $\iota(T_{\chi})$ is $\nu$. Such questions of interpretability tie in closely with the notion of \emph{endogeny} (see Definition 7 and other relevant parts of \S~2.4 of \cite{aldous2005a}), some literature on which has been briefly discussed below. 

Perhaps one of the most seminal works in the literature, that relates directly to the approach involving recurrence relations that we adopt for analyzing the games, is \cite{aldous2005a}, which concerns itself with the study of recursive distributional equations. In \cite{aldous2005a}, a family $(\xi_{i}: i \geqslant 1)$ of jointly distributed random variables and a family $(X_{i}: i \geqslant 1)$ of independent and identically distributed random variables are considered, with the law of each $X_{i}$ being $\mu$. The two families $(\xi_{i}: i \geqslant 1)$ and $(X_{i}: i \geqslant 1)$ are independent of each other. Letting $T(\mu)$ indicate the law of the random variable $g\big((\xi_{i}, X_{i}): i \geqslant 1\big)$, where $g$ is a function defined on a suitable domain, one of \cite{aldous2005a}'s chief goals is to explore the existence and uniqueness of those $\mu$, termed \emph{fixed points}, for which we have $T(\mu) = \mu$. In this context, the \emph{recursive tree process} is introduced (see \S~2.3 of \cite{aldous2005a}). Let a vertex $u$ of a rooted tree have $N$ children (where $N$ could be a random number taking values in $\mathbb{N}_{0} \cup \{\infty\}$) that are named $u_{1}, u_{2}, \ldots, u_{N}$ when $N$ is finite, and $u_{1}, u_{2}, u_{3}, \ldots$ when $N$ takes the value $\infty$. Let $X_{i}$ denote the (random) state assigned to the child $u_{i}$, and let $\xi_{i}$ indicate some random noise associated with $u_{i}$, for every $i$. Then the (random) state $X$ assigned to the parent vertex $u$ is given by $g\big((\xi_{i}, X_{i}): i \geqslant 1\big)$. Further advancements in the study of recursive distributional equations and recursive tree processes (as well as the notion of endogeny that we alluded to in the previous paragraph) can be found in \cite{bandyopadhyay2002bivariate}, \cite{bandyopadhyay2006necessary}, \cite{bandyopadhyay2004bivariate}, \cite{bandyopadhyay2011endogeny} etc. 

We mention here that recursive distributional equations find numerous applications, such as in finding \emph{maximal weight partial matchings on random trees} (see \S 3 of \cite{aldous2004objective}) and \emph{minimal cost perfect matchings} (see \S 5 of \cite{aldous2004objective}), in the \emph{cavity method} implemented for \emph{community detection in sparse graphs} (see, for instance, \S 3.3 of \cite{javanmard2016phase}), and in establishing \emph{correlation decay} that subsequently leads to the discovery of the asymptotics for the \emph{log-partition functions} in certain models of statistical physics (see, for instance, the discussions in \cite{bandyopadhyay2008counting}), to name just a few. 

It is undeniable that a tremendous motivation for studying percolation games arises from the classically studied topic of percolation in physics. It is the presence of players who play \emph{optimally} and make \emph{adversarial} moves against one another, in a bid to win, that sets the former apart from the latter, and we take this opportunity to draw an important comparison between the two. It is well-known and relatively straightforward to see that if each edge of the rooted $d$-regular tree $T_{d}$, independently, is marked \emph{closed} (this corresponds to the label that reads trap in our set-up) with probability $p$, and \emph{open} (this corresponds to the label that reads safe in our set-up) with probability $1-p$, then for all $p < 1-d^{-1}$, there exists an open path (i.e.\ a path that comprises \emph{only} open edges) from the root $\phi$ of $T_{d}$ to infinity, i.e.\ \emph{percolation happens}. However, the existence of such a path is \emph{not} enough to guarantee that the bond percolation game, endowed with the parameter-pair $(p,q) = (p,0)$ (i.e.\ no edge is labeled a target) and played on $T_{d}$, will have a strictly positive probability of resulting in a draw. This is why, as can be deduced from Remark~\ref{rem:regular} by setting $q=0$, we require $p < 1 - (d+1)^{d-1}d^{-d}$ for our game to have a strictly positive probability of draw, where we note that $1 - (d+1)^{d-1}d^{-d} < 1 - d^{-1}$. 

We do not delve too deep into the literature on percolation here, as it is vast and extensive, but we draw the reader's attention to \cite{grimmett1999percolation}, to \cite{toom2001contours} (emphasising on the connections explored therein between cellular automata and percolation), and to \cite{lyons1990random} and \cite{lyons1992random} (both of which specifically study percolation on rooted trees, with the former revealing a relation between the \emph{branching number} of the tree, the probability with which each edge is left open, and the probability of existence of an open path from the root to infinity). While we are on the topic of percolation, we refer to the \emph{Maker-Breaker percolation games} (see \cite{day2021maker1}, \cite{day2021maker2}, \cite{dvovrak2021maker}). To begin with, all edges of an infinite connected graph are marked \emph{unsafe}. Players \emph{Maker} and \emph{Breaker} take turns to make moves, where, in each of her moves, Maker marks $a$ of the yet-unsafe edges as \emph{safe}, and in each of her moves, Breaker deletes $b$ of the yet-unsafe edges from the graph, where $a \in \mathbb{N}$ and $b \in \mathbb{N}$ are pre-specified. Breaker wins if at any point of time in the game, the connected component containing a pre-specified vertex of the graph becomes finite; else Maker wins.

Finally, we mention here that the ergodicity of probabilistic tree automata, as defined in Definition~\ref{defn:ergodicity}, is reminiscent of how \emph{weak spatial mixing} is defined and proved for models of statistical mechanics on rooted trees. In particular, we refer the reader to \S 3.4 of \cite{brightwell2002random}, \cite{jonasson2002uniqueness}, and Lemma 1.3 (and Equation (3) therein) of \cite{de2023uniqueness}. For the definition of \emph{Gibbs states}, corresponding to a given model of statistical mechanics, on an infinite graph, we refer the reader to \cite{simon2014statistical}). From this definition, it becomes evident that for there to be a \emph{unique Gibbs measure} for a model of statistical mechanics defined on a rooted infinite tree, the effect that \emph{any} configuration of \emph{states} or \emph{spins}, assigned to the vertices at distance $n$ from the root $\phi$, has on the induced (random) state at $\phi$ must decay as $n \rightarrow \infty$. This is precisely what \eqref{corr_decay} of Definition~\ref{defn:ergodicity} represents, as well.

\subsection{Organization of the paper}\label{subsec:org} The proof of Theorem~\ref{thm:main_GW_bond} is described in detail in \S\ref{sec:bond_GW}, whereas the proof of Theorem~\ref{thm:main_GW_site}, both directly and via an intimate connection between bond percolation games and site percolation games on rooted GW trees, is outlined in \S\ref{sec:site_GW}. The whole of \S\ref{sec:gen_tech} is dedicated to explaining the general argument that is subsequently employed to obtain complete descriptions of the regimes in which the bond percolation game, played on $T_{\chi}$ with $\chi$ belonging to specified classes of distributions, has probability $0$ of resulting in a draw. The results corresponding to the Binomial, Poisson and Negative Binomial offspring distributions, as well as for any offspring distribution supported on $\{0,d\}$ for $d \in \mathbb{N}$, $d \geqslant 2$, are stated in Theorems~\ref{thm:binomial_bond},\ref{thm:Poisson_bond},\ref{thm:neg_bin_bond} and \ref{thm:0_d_bond}, and the corresponding proofs can be found in  \S\ref{sec:binomial_bond}, \S\ref{sec:Poisson_bond}, \S\ref{sec:negative_binomial_bond} and \S\ref{sec:0_d_bond} respectively. In \S\ref{sec:ergodicity}, we illustrate how a suitably defined probabilistic tree automaton is ergodic if and only if the bond percolation game played on $T_{d}$, where $T_{d}$ denotes the rooted $d$-regular tree in which each vertex has precisely $d$ children, has probability $0$ of resulting in a draw. Finally, \S\ref{sec:finite_expected_duration} is dedicated to the proof of Theorem~\ref{thm:avg_dur} which talks about the expected duration of a bond percolation game that has probability $0$ of ending in a draw.

\section{The bond percolation game on the rooted Galton-Watson tree}\label{sec:bond_GW}
Recall that the players are named $A$ and $B$, and $A$ is assumed to play the first round of the game. Let us define the following (random) subsets of the vertices of $T_{\chi}$:
\begin{enumerate}
\item We indicate by $W$ the subset comprising vertices $v$ of $T_{\chi}$ such that if the game begins with the token at $v$, then $A$ wins.
\item We indicate by $L$ the subset comprising vertices $v$ of $T_{\chi}$ such that if the game begins with the token at $v$, then $A$ loses.
\item We indicate by $D$ the subset comprising vertices $v$ of $T_{\chi}$ such that if the game begins with the token at $v$, it results in a draw.
\end{enumerate}
We set $W_{0} = L_{0} = \emptyset$, and for each $n \in \mathbb{N}$, we let 
\begin{enumerate}
\item $W_{n} \subset W$ consist of vertices $v$ such that the game that begins with the token at $v$ lasts for at most $n$ rounds (i.e.\ $A$ can guarantee to win the game within $n$ rounds);
\item $L_{n} \subset L$ consist of vertices $v$ such that the game that begins with the token at $v$ lasts for at most $n$ rounds (i.e.\ $B$ can guarantee to defeat $A$ within $n$ rounds);
\item $D_{n}$ consist of vertices $v$ such that the game that begins with the token at $v$ lasts for at least $n+1$ rounds.
\end{enumerate} 
We mention here, in particular, that $v \in L_{1}$ if and only if either $v$ is childless or the directed edge $(v,v')$ has been labeled a trap for \emph{every} child $v'$ of $v$; on the other hand, $v \in W_{1}$ if and only if there exists at least one child $v'$ of $v$ such that $(v,v')$ has been labeled a target.

We let $w = w(p,q)$ denote the probability of the event that the root $\phi$ of $T_{\chi}$ belongs to $W$, and $\ell = \ell(p,q)$ the probability of the event that $\phi$ belongs to $L$ (henceforth, we drop $(p,q)$ from the notation as the parameter-pair $(p,q)$ has already been fixed before the game begins). We let $w_{n}$ denote the probability of the event that $\phi$ is in $W_{n}$, and $\ell_{n}$ denote the probability of the event that $\phi$ is in $L_{n}$, for each $n \in \mathbb{N}_{0}$ (so that $w_{0} = \ell_{0} = 0$). From the definitions above, it is immediate that $W_{n} \subset W_{n+1}$ and $L_{n} \subset L_{n+1}$ for all $n \in \mathbb{N}_{0}$, and consequently, each of $\{w_{n}\}_{n \in \mathbb{N}_{0}}$ and $\{\ell_{n}\}_{n \in \mathbb{N}_{0}}$ forms an increasing sequence of non-negative real numbers in $[0,1]$.

Our objective, now, is to show the following: if a game is destined to \emph{not} result in a draw, the player who is destined to win can guarantee to do so in a \emph{finite} number of rounds that can be specified \emph{a priori}. Formally, this is captured by the following lemma: 
\begin{lemma}\label{lem:compactness}
In the bond percolation game played on $T_{\chi}$, we have $\bigcup_{n=0}^{\infty}W_{n} = W$ and $\bigcup_{n=0}^{\infty}L_{n} = L$.
\end{lemma}
This lemma is analogous to, and its proof the same as that of, Proposition 7 of \cite{holroyd2021galton}. We, therefore, have omitted the proof in this paper. Note that Lemma~\ref{lem:compactness} allows us to conclude that 
\begin{equation}
w_{n} \uparrow w \text{ and } \ell_{n} \uparrow \ell \text{ as } n \rightarrow \infty.\label{w_{n}_ell_{n}_limits}
\end{equation}

\subsection{Recurrence relations for the bond percolation game}\label{subsec:recurrence_bond}
Given that the root $\phi$ of $T_{\chi}$ has $m$ children, say $v_{1}, \ldots, v_{m}$, for some $m \in \mathbb{N}$, we note that $\phi \in W_{n}$ if and only if 
\begin{enumerate}
\item either there exists some $1 \leqslant k \leqslant m$ such that the edge $(\phi,v_{k})$ is a target,
\item or else, there exists no $1 \leqslant \ell \leqslant m$ such that $(\phi,v_{\ell})$ is a target, but there exists at least one $1 \leqslant k \leqslant m$ such that $(\phi,v_{k})$ is safe and $v_{k} \in L_{n-1}$.
\end{enumerate}
Note that, conditioned on $\phi$ having children $v_{1}, \ldots, v_{m}$, the (random) trees $T_{\chi}(v_{1})$, $\ldots$, $T_{\chi}(v_{m})$ are all i.i.d.\ copies of $T_{\chi}$ (recall that $T_{\chi}(v_{i})$ denotes the sub-tree of $T_{\chi}$ induced on the subset of vertices comprising $v_{i}$ and all its descendants), so that the probability of the event $v_{k} \in L_{n-1}$ is given by $\ell_{n-1}$, for each $1 \leqslant k \leqslant m$. We thus have
\begin{align}
w_{n} &= \sum_{m=0}^{\infty} \left\{1 - (1-q)^{m}\right\}\chi(m) + \sum_{m=0}^{\infty}\sum_{i=0}^{m}\left\{(1-q)^{m} - p^{i}(1-q)^{m-i}\right\}\Prob\left[i \text{ out of } v_{1}, \ldots, v_{m} \text{ are in } L_{n-1}\right]\chi(m)\nonumber\\
&= \sum_{m=0}^{\infty}\sum_{i=0}^{m}\left\{1 - p^{i}(1-q)^{m-i}\right\}{m \choose i}\ell_{n-1}^{i}\left(1-\ell_{n-1}\right)^{m-i}\chi(m)\nonumber\\
&= 1 - \sum_{m=0}^{\infty}\sum_{i=0}^{m}{m \choose i}\left(p \ell_{n-1}\right)^{i}\left\{(1-q)\left(1-\ell_{n-1}\right)\right\}^{m-i}\chi(m)\nonumber\\
&= 1 - \sum_{m=0}^{\infty}\left\{(1-q) - (1-p-q)\ell_{n-1}\right\}^{m}\chi(m) = 1 - G\left((1-q) - (1-p-q)\ell_{n-1}\right).\label{GW_recurrence_1}
\end{align}
On the other hand, given that $\phi$ has children $v_{1}, \ldots, v_{m}$, it belongs to $L_{n}$ if and only if 
\begin{enumerate}
\item the edge $(\phi,v_{\ell})$ is \emph{not} a target for \emph{any} $1 \leqslant \ell \leqslant m$,
\item for every $v_{\ell} \in L_{n-1}$, the edge $(\phi,v_{\ell})$ is a trap (since otherwise, $A$ can move the token, in the first round, from $\phi$ to a $v_{\ell}$ such that $v_{\ell} \in L_{n-1}$ and $(\phi,v_{\ell})$ is safe, and thus win the game),
\item for every $v_{\ell} \in D_{n-1}$, the edge $(\phi,v_{\ell})$ is a trap (since otherwise, $A$ can move the token, in the first round, from $\phi$ to a $v_{\ell}$ such that $v_{\ell} \in D_{n-1}$ and $(\phi,v_{\ell})$ is safe, and thus make the game last at least $n+1$ rounds).
\end{enumerate}
Note that $\phi \in L_{n}$ automatically if $\phi$ is childless. We thus have
\begin{align}
\ell_{n} &= \sum_{m=0}^{\infty}\sum_{i,j \in \mathbb{N}_{0}:i+j \leqslant m}p^{i+j}(1-q)^{m-i-j}\Prob\left[i \text{ children are in } L_{n-1} \text{ and } j \text{ in } D_{n-1}, \text{ out of } v_{1}, \ldots, v_{m}\right]\chi(m)\nonumber\\
&= \sum_{m=0}^{\infty}\sum_{i,j \in \mathbb{N}_{0}:i+j \leqslant m} p^{i+j}(1-q)^{m-i-j} {m \choose i} {m-i \choose j} \ell_{n-1}^{i} \left(1-\ell_{n-1}-w_{n-1}\right)^{j} w_{n-1}^{m-i-j} \chi(m)\nonumber\\
&= \sum_{m=0}^{\infty}\sum_{i,j \in \mathbb{N}_{0}:i+j \leqslant m}\frac{m!}{i! j! (m-i-j)!}\left(p \ell_{n-1}\right)^{i} \left(p - p \ell_{n-1} - p w_{n-1}\right)^{j} \left\{(1-q)w_{n-1}\right\}^{m-i-j}\chi(m)\nonumber\\
&= \sum_{m=0}^{\infty}\left\{p + (1-p-q)w_{n-1}\right\}^{m}\chi(m) = G\left(p + (1-p-q)w_{n-1}\right).\label{GW_recurrence_2}
\end{align}  

Combining \eqref{GW_recurrence_1} and \eqref{GW_recurrence_2}, and introducing the change of variable $w'_{n} = p + (1-p-q)w_{n}$ for each $n \in \mathbb{N}_{0}$, we obtain
\begin{align}
w'_{n+1} &= p + (1-p-q) - (1-p-q)G\left((1-q) - (1-p-q)G\left(p + (1-p-q)w_{n-1}\right)\right)\nonumber\\
&= (1-q) - (1-p-q)G\left((1-q) - (1-p-q)G\left(w'_{n-1}\right)\right) = g_{p,q}^{(2)}(w'_{n-1}),\label{GW_recurrence_3}
\end{align}
where the function $g_{p,q}$ is as defined in \eqref{GW_g_defn}. Taking the limits of both sides of \eqref{GW_recurrence_3}, as $n \rightarrow \infty$, and using \eqref{w_{n}_ell_{n}_limits}, we see that $w' = p + (1-p-q)w$ is a fixed point of $g_{p,q}^{(2)}$ in $[0,1]$. 

\subsection{Proofs of parts \eqref{main_GW_bond_1} and \eqref{main_GW_bond_2} of Theorem~\ref{thm:main_GW_bond}}\label{subsec:parts_1_2_bond_GW}
First, introducing the change of variable $\ell'_{n} = (1-p-q)\ell_{n}$, for every $n \in \mathbb{N}_{0}$, and using \eqref{GW_recurrence_2} and \eqref{GW_g_defn}, we have
\begin{align}
\ell'_{n} = (1-p-q)\ell_{n} = (1-p-q)G\left(w'_{n-1}\right) = (1-q) - g_{p,q}\left(w'_{n-1}\right).\label{GW_recurrence_4}
\end{align}
Taking the limits of both sides of \eqref{GW_recurrence_4} and using \eqref{w_{n}_ell_{n}_limits}, we see that $\ell' = (1-p-q)\ell$ satisfies the relation $\ell' = (1-q) - g_{p,q}(w')$. This proves part \eqref{main_GW_bond_2} of Theorem~\ref{thm:main_GW_bond}.

Next, we prove part \eqref{main_GW_bond_1} of Theorem~\ref{thm:main_GW_bond},  i.e.\ that $w'$ equals the minimum positive fixed point of $g_{p,q}^{(2)}$. To begin with, we make a couple of observations that are going to be useful to us throughout this paper (even where we consider specific probability distributions in place of $\chi$). The first of these provides strictly positive lower bounds on $g_{p,q}^{(2)}(0)$, for various values of $(p,q)$. Note that 
\begin{align}
g_{p,q}^{(2)}(0) = (1-q) - (1-p-q)G[(1-q) - (1-p-q)G(0)].\nonumber
\end{align}
When $p > 0$, since $G[(1-q) - (1-p-q)G(0)] \leqslant 1$, we have 
\begin{align}
g_{p,q}^{(2)}(0) \geqslant (1-q) - (1-p-q) = p > 0.\label{g_{p,q}^{(2)}(0)_lower_bound_1}
\end{align}
On the other hand, when $p=0$, we have $1 > q > 0$ (since $(p,q) \in I$, with $I$ as defined in \eqref{I_defn}), so that
\begin{align}
g_{p,q}^{(2)}(0) &= (1-q) - (1-q)G[(1-q) - (1-q)G(0)] = (1-q)\left\{1 - G[(1-q)(1-\chi(0))]\right\},\label{g_{p,q}^{(2)}(0)_lower_bound_2}
\end{align}
and this is strictly positive because $1-q > 0$ and $G[(1-q)(1-\chi(0))] < 1$ (since $(1-q)(1-\chi(0)) < 1$, because $1-q < 1$). The second observation provides an upper bound on $g_{p,q}^{(2)}(1)$ that is strictly less than $1$, as follows:
\begin{align}
g_{p,q}^{(2)}(1) = (1-q) - (1-p-q)G[(1-q) - (1-p-q)G(1)] = (1-q) - (1-p-q)G(p).\nonumber
\end{align}
When $p > 0$, we have $G(p) > 0$, and by \eqref{I_defn}, we have $1-p-q > 0$. Therefore, in this case, 
\begin{align}
g_{p,q}^{(2)}(1) \leqslant 1 - (1-p-q)G(p) < 1.\label{g_{p,q}^{(2)}(1)_upper_bound_1}
\end{align}
When $p=0$, we must have $q > 0$ (since $(p,q) \in I$, with $I$ as defined in \eqref{I_defn}). Moreover, $(1-p-q)G(p) = (1-q)\chi(0) \geqslant 0$. Thus
\begin{align}
g_{p,q}^{(2)}(1) \leqslant 1-q < 1.\label{g_{p,q}^{(2)}(1)_upper_bound_2}
\end{align}
Note, crucially, that the bounds obtained from \eqref{g_{p,q}^{(2)}(0)_lower_bound_1}, \eqref{g_{p,q}^{(2)}(0)_lower_bound_2}, \eqref{g_{p,q}^{(2)}(1)_upper_bound_1} and \eqref{g_{p,q}^{(2)}(1)_upper_bound_2} together make sure that \emph{neither} $0$ \emph{nor} $1$ is a fixed point of $g_{p,q}^{(2)}$ as long as $(p,q) \in I$ with $I$ as defined in \eqref{I_defn}.

Let $\gamma \in (0,1)$ be a fixed point of $g_{p,q}^{(2)}$. Note that 
\begin{align}
\frac{d}{dx}g_{p,q}^{(2)}(x) &= (1-p-q)^{2}G'\left((1-q) - (1-p-q)G(x)\right) G'(x),\label{GW_g_derivative}
\end{align}
and this is strictly positive on $(0,1]$ since $G'(x) = \sum_{m=1}^{\infty}m \chi(m) x^{m-1} > 0$ for all $x \in (0,1]$ (this strict inequality is ensured by our assumption that $0 \leqslant \chi(0) < 1$, which implies that there exists at least one $m \in \mathbb{N}$ with $\chi(m) > 0$). Therefore, $g_{p,q}^{(2)}$ is strictly increasing on $[0,1]$, so that using \eqref{g_{p,q}^{(2)}(0)_lower_bound_1} and recalling that $w_{0}' = p + (1-p-q)w_{0} = p$ since $w_{0} = 0$, we have 
\begin{align}
\gamma &= g_{p,q}^{(2)}(\gamma) > g_{p,q}^{(2)}(0) \geqslant p = w'_{0}.\nonumber 
\end{align} 
We iteratively apply $g_{p,q}^{(2)}$ to both sides of the above inequality, and make use of the increasing nature of $g_{p,q}^{(2)}$ as well as \eqref{GW_recurrence_3}, leading to
\begin{align}
\gamma = \left(g_{p,q}^{(2)}\right)^{(n)}(\gamma) \geqslant \left(g_{p,q}^{(2)}\right)^{(n)}(w'_{0}) = w'_{2n}.\nonumber
\end{align}
Upon taking the limits of both sides as $n \rightarrow \infty$, and using \eqref{w_{n}_ell_{n}_limits}, we deduce that $\gamma \geqslant w'$. This shows that, indeed, $w'$ is the minimum positive fixed point of $g_{p,q}^{(2)}$.

\subsection{Proof of part \eqref{main_GW_bond_3} of Theorem~\ref{thm:main_GW_bond}}\label{subsec:part_3_GW_bond}
The probability of the event that the root $\phi$ belongs to $D$ is $0$, i.e.\ the probability of the event that the bond percolation game that starts from the root of $T_{\chi}$ results in a draw is $0$, if and only if 
\begin{align}
w + \ell = 1 \Longleftrightarrow w'+\ell' = (1-q) \Longleftrightarrow w' = g_{p,q}(w'),\label{GW_draw_probab_0_cond}
\end{align}
i.e.\ $w'$ is a fixed point of $g_{p,q}$. Note, from \eqref{GW_g_defn}, that the derivative $g'_{p,q}(x) = -(1-p-q)G'(x)$ is strictly negative for all $x \in (0,1]$, so that $g_{p,q}$ is strictly decreasing on $[0,1]$. Note that $g_{p,q}(0) > p \geqslant 0$ since $0 \leqslant G(0) = \chi(0) < 1$, whereas $g_{p,q}(1) = p < 1$. Therefore, the curve $y=g_{p,q}(x)$ lies above the line $y=x$ at $x=0$, and below the line $y=x$ at $x=1$. We conclude, therefore, that $g_{p,q}$ has a \emph{unique} fixed point in $[0,1]$, and henceforth, we denote it by $\alpha_{p,q}$. It is obvious that $\alpha_{p,q}$ is also a fixed point of $g_{p,q}^{(2)}$. From \eqref{GW_draw_probab_0_cond}, we conclude that the probability of draw is $0$ if and only if $w' = \alpha_{p,q}$.

Recall, from above, that we have already shown $w'$ to be the smallest positive fixed point of $g_{p,q}^{(2)}$. We now state and prove the following lemma that is key to much of our analysis in this paper:
\begin{lemma}\label{lem:draw_probab_0_only_if_unique_fixed_point}
The minimum positive fixed point $w'$ of $g_{p,q}^{(2)}$ equals $\alpha_{p,q}$, the unique fixed point of $g_{p,q}$ in $[0,1]$, if and only if $g_{p,q}^{(2)}$ has a unique fixed point in $[0,1]$.
\end{lemma}
\begin{proof}
We prove only the implication that is not obvious. As $w'$ is the minimum positive fixed point of $g_{p,q}^{(2)}$, if $\beta \in (0,1)$ is \emph{any other} fixed point of $g_{p,q}^{(2)}$, we must have $w' < \beta$. If, furthermore, $w'$ equals $\alpha_{p,q}$, then we must have $w' = \alpha_{p,q} < \beta$. Since $g_{p,q}$ is strictly decreasing on $[0,1]$, as shown above, we have 
\begin{equation}
w' = \alpha_{p,q} = g_{p,q}(\alpha_{p,q}) > g_{p,q}(\beta).\label{eq:lem_unique_1}
\end{equation}
Note that, since $\beta$ is a fixed point of $g_{p,q}^{(2)}$,
\begin{equation}
\beta = g_{p,q}^{(2)}(\beta) \implies g_{p,q}(\beta) = g_{p,q}\left(g_{p,q}^{(2)}(\beta)\right) = g_{p,q}^{(2)}(g_{p,q}(\beta)).\label{eq:lem_unique_2}
\end{equation} 
By \eqref{eq:lem_unique_2}, it is evident that $g_{p,q}(\beta)$ is a fixed point of $g_{p,q}^{(2)}$. We emphasize to the reader here that $g_{p,q}(\beta)$ is strictly positive because, by \eqref{GW_g_defn}, 
\begin{align}
g_{p,q}(\beta) = (1-q) - (1-p-q)G(\beta) \geqslant (1-q) - (1-p-q) = p > 0 \text{ when } p > 0,\nonumber
\end{align}
whereas when $p=0$, we have $0 < q < 1$ (by definition of $I$ in \eqref{I_defn}), and $\beta < 1 \implies G(\beta) < 1$, so that
\begin{align}
g_{p,q}(\beta) = (1-q) - (1-q)G(\beta) = (1-q)[1-G(\beta)] > 0.\nonumber
\end{align}
Thus, $g_{p,q}(\beta)$ is a positive fixed point of $g_{p,q}^{(2)}$, and by \eqref{eq:lem_unique_1}, it is strictly less than $w'$, the smallest positive fixed point of $g_{p,q}^{(2)}$, thus leading to a contradiction. This completes the proof.
\end{proof}

\section{The site percolation game on the rooted Galton-Watson tree}\label{sec:site_GW}
There are two ways to analyze the site percolation game on $T_{\chi}$. The first is of the same flavour as the approach adopted in \S\ref{sec:bond_GW}, and we outline it briefly below, whereas the second has been detailed in \S\ref{subsec:bond_site_connection}. Assuming that $A$ plays the first round of the game, we let 
\begin{enumerate}
\item $\widetilde{W}$ denote the subset comprising those vertices $v$ of $T_{\chi}$ such that the game starting with the token at $v$ is won by $A$,
\item $\widetilde{L}$ denote the subset comprising those vertices $v$ of $T_{\chi}$ such that the game starting with the token at $v$ is lost by $A$,
\item $\widetilde{D}$ denote the subset comprising those vertices $v$ of $T_{\chi}$ such that the game starting with the token at $v$ results in a draw.
\end{enumerate}
In particular, if the vertex $v$ has been labeled a trap, we let $v \in \widetilde{W}$, and if $v$ has been labeled a target, we let $v \in \widetilde{L}$. The intuition behind these conventions is as follows: we imagine an \emph{unseen} round that takes place before the actual game (which starts with the token at $v$) begins, during which player $B$ moves the token from somewhere else to the vertex $v$. If $v$ is a trap, $B$ loses immediately, allowing $A$ to win, and if $v$ is a target, then $B$ wins immediately, forcing $A$ to lose even before the actual game truly begins. 

We let $\widetilde{W}_{0} = \widetilde{L}_{0} = \emptyset$, and for every $n \in \mathbb{N}$, we define $\widetilde{W}_{n} \subset \widetilde{W}$ and $\widetilde{L}_{n} \subset \widetilde{L}$ to be the subsets comprising all those vertices $v$ of $T_{\chi}$ such that the game that begins with the token at $v$ lasts for less than $n$ rounds, while $\widetilde{D}_{n}$ denotes the subset comprising all those vertices $v$ of $T_{\chi}$ such that the game that begins with the token at $v$ lasts for at least $n$ rounds. Note the crucial difference between the definition of $\widetilde{W}_{n}$ and that of $W_{n}$ in \S\ref{sec:bond_GW} (likewise, between $L_{n}$ and $\widetilde{L}_{n}$, and between $D_{n}$ and $\widetilde{D}_{n}$) -- for $v \in W_{n}$, the game that begins with the token at $v$ lasts for at most $n$ rounds, whereas for $v \in \widetilde{W}_{n}$, the game that begins with the token at $v$ lasts for at most $n-1$ rounds. 

We let $\widetilde{w}_{n}$, $\widetilde{w}$, $\widetilde{\ell}_{n}$ and $\widetilde{\ell}$ denote the probabilities of the root $\phi$ of $T_{\chi}$ belonging to $\widetilde{W}_{n}$, $\widetilde{W}$, $\widetilde{L}_{n}$ and $\widetilde{L}$ respectively. A lemma analogous to Lemma~\ref{lem:compactness} yields $\widetilde{w}_{n} \uparrow \widetilde{w}$ and $\widetilde{\ell}_{n} \uparrow \widetilde{\ell}$ as $n \rightarrow \infty$.

For $n \in \mathbb{N}$, the root $\phi$ belongs to $\widetilde{W}_{n}$ if and only if either $\phi$ is itself a trap, or it is safe \emph{and} has at least one child that belongs to $\widetilde{L}_{n-1}$. Consequently,
\begin{align}
\widetilde{w}_{n} &= p + (1-p-q)\sum_{m=0}^{\infty}\left\{1 - \left(1-\widetilde{\ell}_{n-1}\right)^{m}\right\}\chi(m) = (1-q) - (1-p-q)G\left(1-\widetilde{\ell}_{n-1}\right).\label{site_recurrence_1}
\end{align} 
The root $\phi$ belongs to $\widetilde{L}_{n}$ if and only if either $\phi$ itself is a target, or it is safe \emph{and} each of its children is in $W_{n-1}$, leading to
\begin{align}
\widetilde{\ell}_{n} &= q + (1-p-q)\sum_{m=0}^{\infty}\widetilde{w}_{n-1}^{m}\chi(m) = q + (1-p-q)G\left(\widetilde{w}_{n-1}\right).\label{site_recurrence_2}
\end{align}
Combining \eqref{site_recurrence_1} and \eqref{site_recurrence_2}, and using \eqref{GW_g_defn}, we obtain: $\widetilde{w}_{n+1} = g_{p,q}^{(2)}\left(\widetilde{w}_{n-1}\right)$, for each $n \in \mathbb{N}$. Arguing as in \S\ref{subsec:parts_1_2_bond_GW}, we conclude that $\widetilde{w} = \lim_{n \rightarrow \infty}\widetilde{w}_{n}$ is the smallest positive fixed point of $g_{p,q}^{(2)}$, and \eqref{site_recurrence_2}) yields $\widetilde{\ell} = q + (1-p-q)G(\widetilde{w})$. This completes the proofs of parts \eqref{main_GW_site_1} and \eqref{main_GW_site_2} of Theorem~\ref{thm:main_GW_site}.

The probability of draw is $0$ if and only if $\widetilde{w} + \widetilde{\ell} = 1$, which is equivalent to $\widetilde{w} = g_{p,q}\left(\widetilde{w}\right)$. By Lemma~\ref{lem:draw_probab_0_only_if_unique_fixed_point}, we conclude that this is possible if and only if $g_{p,q}^{(2)}$ has a unique fixed point in $[0,1]$, completing the proof of part \eqref{main_GW_site_3} of Theorem~\ref{thm:main_GW_site}. This, in turn, tells us that the set of values of the parameter-pair $(p,q)$ for which the probability of draw in the bond percolation game is $0$ is precisely the set of values of $(p,q)$ for which the probability of draw in the site percolation game is $0$, when both are played on the rooted GW tree $T_{\chi}$.

\subsection{A different method for analyzing the site percolation game}\label{subsec:bond_site_connection}
The second approach for analyzing the site percolation game played on $T_{\chi}$ relies on an intimate connection that this game bears with the bond percolation game played on $T_{\chi}$. We denote by $(T,\omega)$ a realization $T$ of $T_{\chi}$ in which $\omega((u,v))$ denotes the label (trap, target or safe) assigned to $(u,v)$ for every directed edge $(u,v)$ of $T$. Given a $(T,\omega)$, we define a (deterministic) assignment $\widetilde{\omega}$ of trap / target / safe labels to the vertices of $T$, as follows: the root $\phi$ of $T$ is labeled safe under $\widetilde{\omega}$, and given any vertex $v$ of $T$ that is not the root $\phi$, letting $u$ be the parent of $v$, 
\begin{enumerate}
\item $v$ is labeled a trap under $\widetilde{\omega}$ if the directed edge $(u,v)$ has been labeled a trap under $\omega$, 
\item $v$ is labeled a target under $\widetilde{\omega}$ if the directed edge $(u,v)$ has been labeled a target under $\omega$,
\item $v$ is labeled safe under $\widetilde{\omega}$ if the directed edge $(u,v)$ has been labeled safe under $\omega$.
\end{enumerate} 
The way we have constructed $\widetilde{\omega}$ allows us to conclude that $\phi$ belongs to $\widetilde{W}$ under $\widetilde{\omega}$ if and only if it belongs to $W$ under $\omega$, and $\phi$ belongs to $\widetilde{L}$ under $\widetilde{\omega}$ if and only if it belongs to $L$ under $\omega$.

From above, we conclude that the event $\phi \in \widetilde{W}$ (in the site percolation game) has the same probability as the union of the following two disjoint events:
\begin{enumerate}
\item $\phi$ has been labeled a trap, which happens with probability $p$,
\item $\phi$ is safe and $\phi \in W$ (in the bond percolation game), which happens with probability $(1-p-q) w$,
\end{enumerate}
thus yielding $\widetilde{w} = p + (1-p-q)w = w'$ (where $w'$ is the transformation of $w$ as defined in the statement of Theorem~\ref{thm:main_GW_bond}). Likewise, the event $\phi \in \widetilde{L}$ (in the site percolation game) has the same probability as the union of the following two disjoint event:
\begin{enumerate}
\item $\phi$ has been labeled a target, which happens with probability $q$,
\item $\phi$ is safe and $\phi \in L$ (in the bond percolation game), which happens with probability $(1-p-q)\ell$,
\end{enumerate}
thus yielding $\widetilde{\ell} = q + (1-p-q)\ell = q + \ell' = 1 - g_{p,q}(w') = q + (1-p-q)G(\widetilde{w})$, from Theorem~\ref{thm:main_GW_bond} and using \eqref{GW_g_defn} and the fact, established above, that $\widetilde{w} = w'$. Finally, the event that $\phi \in \widetilde{D}$ (in the site percolation game) is the same as the event that $\phi$ is safe and $\phi \in D$ (in the bond percolation game), thus implying that the probability of draw in the site percolation game is $0$ if and only if the probability of draw in the bond percolation game is $0$. This completes the proofs of all parts of Theorem~\ref{thm:main_GW_site}.

\section{When $\chi$ belongs to one of several specified classes of distributions}\label{sec:specific_offspring}
\subsection{A general outline for the proof techniques adopted in the subsequent subsections}\label{sec:gen_tech}
In the sections that follow, we focus on the probability of draw in the bond percolation game played on $T_{\chi}$, where $\chi$ is specified to be one of the most commonly studied probability distributions supported on $\mathbb{N}_{0}$. In each case, a set of rather \emph{ad hoc} computations are necessary, but broadly speaking, the underlying argument remains the same, and this is what we outline in \S\ref{sec:gen_tech}.

To begin with, we compute the first and second derivatives of the function $g_{p,q}^{(2)}$, with $g_{p,q}$ as defined in \eqref{GW_g_defn} (with the probability generating function $G$ now being that of the $\chi$ under consideration), and in each case, we find that there exists some unique $x_{p,q} > 0$ such that
\begin{enumerate}
\item $g_{p,q}^{(2)}$ is strictly convex, and equivalently, its derivative is strictly increasing, for $0 \leqslant x < x_{p,q}$,
\item $g_{p,q}^{(2)}$ is strictly concave, and equivalently, its derivative is strictly decreasing, for $x > x_{p,q}$.
\end{enumerate}
The maximum of the derivative of $g_{p,q}^{(2)}(x)$, for $x > 0$, is thus attained at $x = x_{p,q}$. It is important to note here (and it will be emphasized in the subsequent sections as well) that the probability generating function $G$ of $\chi$ may be defined on a domain $\mathcal{D} \subset \mathbb{R}^{+}$ that is possibly larger than the interval $[0,1]$ (i.e.\ $[0,1] \subset \mathcal{D}$), but $x_{p,q}$ will always be contained inside $\mathcal{D}$. We can thus write
\begin{equation}
\max_{x \in D}\frac{d}{dx}g_{p,q}^{(2)}(x) = \frac{d}{dx}g_{p,q}^{(2)}(x)\Big|_{x=x_{p,q}}.\label{max_derivative_GW}
\end{equation}

Next, we deduce the condition (involving $p$, $q$ and the underlying parameters of the $\chi$ under consideration) that is equivalent to 
\begin{equation}
\frac{d}{dx}g_{p,q}^{(2)}(x)\big|_{x=x_{p,q}} \leqslant 1.\label{max_derivative_GW_leq_1}
\end{equation}
Note that when \eqref{max_derivative_GW_leq_1} holds, the derivative of $g_{p,q}^{(2)}$, and hence the slope of the curve $y=g_{p,q}^{(2)}(x)$, is bounded above by $1$ throughout $\mathcal{D}$, and hence throughout $[0,1]$. This gives rise to the following lemma:
\begin{lemma}\label{lem:max_derivative_leq_1}
As long as \eqref{max_derivative_GW_leq_1} holds, the function $g_{p,q}^{(2)}$ has a unique fixed point in $[0,1]$.
\end{lemma}
\begin{proof}
As established via \eqref{g_{p,q}^{(2)}(0)_lower_bound_1} and \eqref{g_{p,q}^{(2)}(0)_lower_bound_2}, the curve $y=g_{p,q}^{(2)}(x)$ lies above $y=x$ at $x=0$. Since $w'$ is the smallest positive fixed point of $g_{p,q}^{(2)}$, from \eqref{main_GW_bond_1} of Theorem~\ref{thm:main_GW_bond}, it must be the case that the curve $y=g_{p,q}^{(2)}(x)$ travels from \emph{above} $y=x$ to \emph{beneath} $y=x$ at $x=w'$. If there exists yet another fixed point $\beta \in (0,1)$ of $g_{p,q}^{(2)}$, then $w' < \beta$. Let us assume that there exists no fixed point of $g_{p,q}^{(2)}$ in the sub-interval $(w',\beta)$. This means that the curve $y=g_{p,q}^{(2)}(x)$ must travel from \emph{beneath} $y=x$ to \emph{above} $y=x$ at $x=\beta$, and therefore, the slope of $y=g_{p,q}^{(2)}(x)$ must exceed the slope of $y=x$ at $x=\beta$. But the slope of $y=x$ equals $1$, and by \eqref{max_derivative_GW_leq_1}, the slope of $y=g_{p,q}^{(2)}(x)$ is bounded above by $1$. This leads to a contradiction, implying that the fixed point $\beta$ cannot exist in this case.
\end{proof}

The next major step in the argument that is common to all the subsequent sections is showing that $w' \leqslant x_{p,q}$ for \emph{all} values of $(p,q)$ in $I$, where $I$ is as defined in \eqref{I_defn}. Recalling from \eqref{g_{p,q}^{(2)}(0)_lower_bound_1} and \eqref{g_{p,q}^{(2)}(0)_lower_bound_2} that the curve $y=g_{p,q}^{(2)}(x)$ lies above $y=x$ at $x=0$, and from \eqref{main_GW_bond_1} of Theorem~\ref{thm:main_GW_bond} that $x=w'$ is the smallest positive value of $x$ at which $y=g_{p,q}^{(2)}(x)$ intersects $y=x$, it suffices to show, via the intermediate value theorem (which can be applied here because $g_{p,q}^{(2)}$ is continuous), that $y=g_{p,q}^{(2)}(x)$ lies \emph{beneath} or \emph{on} $y=x$ at $x=x_{p,q}$, i.e.\ that
\begin{equation}\label{w'_leq_x_{p,q}}
g_{p,q}^{(2)}(x_{p,q}) \leqslant x_{p,q} \text{ for all } (p,q) \in I.
\end{equation}

The final step is to show that 
\begin{equation}\label{x_{p,q}<alpha_{p,q}}
x_{p,q} < \alpha_{p,q} \text{ whenever \eqref{max_derivative_GW_leq_1} does \emph{not} hold},
\end{equation}
where we recall, from \S\ref{subsec:part_3_GW_bond}, that $\alpha_{p,q}$ is the unique fixed point of $g_{p,q}$ in $[0,1]$. From \eqref{GW_g_defn}, it is evident that $g_{p,q}$ is strictly decreasing on $[0,1]$, and the curve $y=g_{p,q}(x)$ lies above the line $y=x$ at $x=0$. Since $x=\alpha_{p,q}$ is the first and only value of $x$ (in $[0,1]$) where the curve $y=g_{p,q}(x)$ intersects $y=x$, if we can show that 
\begin{equation}\label{g_{p,q}(x_{p,q})>x_{p,q}}
g_{p,q}(x_{p,q}) > x_{p,q} \text{ whenever \eqref{max_derivative_GW_leq_1} does not hold},
\end{equation}
then it would go to show that \eqref{x_{p,q}<alpha_{p,q}} is true.

Immediately, from \eqref{w'_leq_x_{p,q}} and \eqref{x_{p,q}<alpha_{p,q}}, we conclude that $g_{p,q}^{(2)}$ does \emph{not} have a unique fixed point when \eqref{max_derivative_GW_leq_1} does not hold. This, along with Lemma~\ref{lem:max_derivative_leq_1} and part \eqref{main_GW_bond_3} of Theorem~\ref{thm:main_GW_bond}, allows us to conclude that
\begin{theorem}
The probability of draw in the bond percolation game played on $T_{\chi}$ is $0$ if and only if \eqref{max_derivative_GW_leq_1} holds, where $\chi$ belongs to the list of probability distributions considered in the rest of this paper.
\end{theorem}

\subsection{The bond percolation game played on the Binomial GW tree}\label{sec:binomial_bond}
Let $\chi$ be the Binomial$(d,\pi)$ distribution, so that $\chi(i) = {d \choose i} \pi^{i} (1-\pi)^{d-i}$ for all $0 \leqslant i \leqslant d$, for some $0 < \pi \leqslant 1$ and some $d \in \mathbb{N}$ with $d \geqslant 2$. The corresponding generating function is $G(x) = \left\{\pi x + (1-\pi)\right\}^{d}$ (note that this is well-defined for all $x \in \mathbb{R}^{+}$, i.e.\ in the context of our discussion in \S\ref{sec:gen_tech}, we have $\mathcal{D} = \mathbb{R}^{+}$ in this case). The function $g_{p,q}$ now becomes
\begin{equation}
g_{p,q}(x) = (1-q) - (1-p-q)\left\{\pi x + (1-\pi)\right\}^{d},\label{g_{p,q,bin}}
\end{equation}
so that 
\begin{align}
g'_{p,q}(x) &= -(1-p-q)\pi d \left\{\pi x + (1-\pi)\right\}^{d-1}.\nonumber
\end{align}
This, in turn, yields
\begin{align}
\frac{d}{dx}g_{p,q}^{(2)}(x) &= g'_{p,q}\left(g_{p,q}(x)\right) g'_{p,q}(x) = (1-p-q)^{2}\pi^{2}d^{2}\left[\pi g_{p,q}(x) + (1-\pi)\right]^{d-1} \left\{\pi x + (1-\pi)\right\}^{d-1}\nonumber\\
&=(1-p-q)^{2}\pi^{2}d^{2}\left[\pi (1-q) - \pi (1-p-q) \left\{\pi x + (1-\pi)\right\}^{d} + (1-\pi)\right]^{d-1} \left\{\pi x + (1-\pi)\right\}^{d-1}\nonumber\\
&=(1-p-q)^{2}\pi^{2}d^{2}\left[\left\{1-\pi q - \pi (1-p-q) \left\{\pi x + (1-\pi)\right\}^{d}\right\} \left\{\pi x + (1-\pi)\right\}\right]^{d-1}\nonumber\\
&=(1-p-q)^{2}\pi^{2}d^{2}\left[(1-\pi q)\left\{\pi x + (1-\pi)\right\} - \pi (1-p-q) \left\{\pi x + (1-\pi)\right\}^{d+1}\right]^{d-1}.\nonumber
\end{align}
Defining the function
\begin{equation}
f_{1}(x) = (1-\pi q)\left\{\pi x + (1-\pi)\right\} - \pi (1-p-q) \left\{\pi x + (1-\pi)\right\}^{d+1},\nonumber
\end{equation}
we see that
\begin{equation}
f'_{1}(x) = \pi (1-\pi q) - (d+1) \pi^{2} (1-p-q) \left\{\pi x + (1-\pi)\right\}^{d},\nonumber
\end{equation}
which is strictly positive if and only if
\begin{align}
&(d+1) \pi^{2} (1-p-q) \left\{\pi x + (1-\pi)\right\}^{d} < \pi (1-\pi q) \Longleftrightarrow \left\{\pi x + (1-\pi)\right\} < \left(\frac{1 - \pi q}{(d+1) \pi (1-p-q)}\right)^{\frac{1}{d}}\nonumber\\
& \Longleftrightarrow x < x_{p,q}, \text{ where } x_{p,q} = \frac{1}{\pi}\left(\frac{1 - \pi q}{(d+1) \pi (1-p-q)}\right)^{\frac{1}{d}} + 1 - \frac{1}{\pi}.\label{x_{p,q,bin}} 
\end{align}
Following the reasoning outlined in \S\ref{sec:gen_tech}, we know that the maximum of the derivative of $g_{p,q}^{(2)}(x)$, over $x \in [0,1]$ (in fact, over $x \in \mathcal{D} = \mathbb{R}^{+}$), is attained at $x=x_{p,q}$, and this maximum value equals
\begin{align}
&(1-p-q)^{2}\pi^{2}d^{2}\left[1-\pi q - \pi (1-p-q) \left(\frac{1 - \pi q}{(d+1) \pi (1-p-q)}\right)\right]^{d-1} \left(\frac{1 - \pi q}{(d+1) \pi (1-p-q)}\right)^{\frac{d-1}{d}}\nonumber\\
&=(1-p-q)^{2}\pi^{2}d^{2}\left[1-\pi q - \frac{1-\pi q}{d+1}\right]^{d-1} \left(\frac{1 - \pi q}{(d+1) \pi (1-p-q)}\right)^{\frac{d-1}{d}}\nonumber\\
&=(1-p-q)^{2}\pi^{2}d^{2} \cdot \frac{d^{d-1} (1-\pi q)^{d-1}}{(d+1)^{d-1}} \left(\frac{1 - \pi q}{(d+1) \pi (1-p-q)}\right)^{\frac{d-1}{d}}\nonumber\\
&= (1-p-q)^{\frac{d+1}{d}} \pi^{\frac{d+1}{d}} d^{d+1} (1-\pi q)^{\frac{(d-1)(d+1)}{d}} (d+1)^{-\frac{(d-1)(d+1)}{d}}\nonumber\\
&= \left\{(1-p-q)^{\frac{1}{d}} \pi^{\frac{1}{d}} d (1-\pi q)^{\frac{d-1}{d}} (d+1)^{-\frac{d-1}{d}}\right\}^{d+1}.\nonumber
\end{align}
The expression above is bounded above by $1$ if and only if we have
\begin{align}\label{max_derivative_bounded_by_1_binomial}
&(1-p-q)^{\frac{1}{d}} \pi^{\frac{1}{d}} d (1-\pi q)^{\frac{d-1}{d}} (d+1)^{-\frac{d-1}{d}} \leqslant 1 \Longleftrightarrow (1-p-q) \pi (1-\pi q)^{d-1} \leqslant \frac{(d+1)^{d-1}}{d^{d}},
\end{align}
which is what the inequality in \eqref{max_derivative_GW_leq_1}, mentioned in \S\ref{sec:gen_tech}, boils down to when $\chi$ is Binomial$(d,\pi)$. As argued in \S\ref{sec:gen_tech}, the function $g_{p,q}^{(2)}$ has a unique fixed point whenever \eqref{max_derivative_bounded_by_1_binomial} holds, and consequently, by part \eqref{main_GW_bond_3} of Theorem~\ref{thm:main_GW_bond}, the probability of draw in the bond percolation game is $0$ in this regime.

As outlined in \S\ref{sec:gen_tech}, our next task is to show that $w' \leqslant x_{p,q}$, which we accomplish by showing that \eqref{w'_leq_x_{p,q}} holds, i.e.\ that $g_{p,q}^{(2)}(x_{p,q}) \leqslant x_{p,q}$ for all $(p,q) \in I$, with $I$ as defined in \eqref{I_defn}. From \eqref{g_{p,q,bin}} and \eqref{x_{p,q,bin}}, we see that
\begin{align}
g_{p,q}(x_{p,q}) &= (1-q) - (1-p-q)\left\{\pi x_{p,q} + (1-\pi)\right\}^{d}\nonumber\\
&= (1-q) - (1-p-q) \left(\frac{1 - \pi q}{(d+1) \pi (1-p-q)}\right) = \frac{(d+1)\pi - 1 - d\pi q}{(d+1)\pi},\label{g_{p,q}(x_{p,q})_binomial}
\end{align}
which leads to
\begin{align}
\pi g_{p,q}(x_{p,q}) + (1-\pi) &= \pi \cdot \frac{(d+1)\pi - 1 - d\pi q}{(d+1)\pi} + (1-\pi) = \frac{d(1-\pi q)}{d+1}.\nonumber
\end{align}
Therefore, the condition $g_{p,q}^{(2)}(x_{p,q}) \leqslant x_{p,q}$ becomes equivalent to
\begin{align}
& (1-q) - (1-p-q) \left(\frac{d(1-\pi q)}{d+1}\right)^{d} \leqslant \frac{1}{\pi}\left(\frac{1 - \pi q}{(d+1) \pi (1-p-q)}\right)^{\frac{1}{d}} + 1 - \frac{1}{\pi}\nonumber\\
\Longleftrightarrow & (1-\pi q) - \pi (1-p-q) \left(\frac{d(1-\pi q)}{d+1}\right)^{d} \leqslant \left(\frac{1 - \pi q}{(d+1) \pi (1-p-q)}\right)^{\frac{1}{d}}\nonumber\\
\Longleftrightarrow & 1 \leqslant \frac{\pi (1-p-q) d^{d} (1-\pi q)^{d-1}}{(d+1)^{d}} + \frac{(1-\pi q)^{-\frac{d-1}{d}}}{(d+1)^{\frac{1}{d}} \pi^{\frac{1}{d}} (1-p-q)^{\frac{1}{d}}}\nonumber\\
\Longleftrightarrow & 1 \leqslant \frac{\pi (1-p-q) d^{d} x^{-d}}{(d+1)^{d}} + \frac{x}{(d+1)^{\frac{1}{d}} \pi^{\frac{1}{d}} (1-p-q)^{\frac{1}{d}}}, \text{ where } x = (1-\pi q)^{-\frac{d-1}{d}}.\label{objective_binomial}
\end{align}

Defining the function $f_{2}(x) = \pi (1-p-q) d^{d} x^{-d} (d+1)^{-d} + x (d+1)^{-\frac{1}{d}} \pi^{-\frac{1}{d}} (1-p-q)^{-\frac{1}{d}}$, we see that
\begin{align}
f'_{2}(x) &= - d \pi (1-p-q) d^{d} x^{-d-1} (d+1)^{-d} + (d+1)^{-\frac{1}{d}} \pi^{-\frac{1}{d}} (1-p-q)^{-\frac{1}{d}}\nonumber\\
&= - \pi (1-p-q) d^{d+1} x^{-d-1} (d+1)^{-d} + (d+1)^{-\frac{1}{d}} \pi^{-\frac{1}{d}} (1-p-q)^{-\frac{1}{d}},\nonumber
\end{align}
and this is strictly positive if and only if 
\begin{align}
&(d+1)^{-\frac{1}{d}} \pi^{-\frac{1}{d}} (1-p-q)^{-\frac{1}{d}} > \pi (1-p-q) d^{d+1} x^{-d-1} (d+1)^{-d}\nonumber\\
\Longleftrightarrow & x^{d+1} > \pi^{\frac{d+1}{d}} (1-p-q)^{\frac{d+1}{d}} d^{d+1} (d+1)^{-\frac{(d-1)(d+1)}{d}}\nonumber\\
\Longleftrightarrow & x > \pi^{\frac{1}{d}} (1-p-q)^{\frac{1}{d}} d (d+1)^{-\frac{d-1}{d}}.\nonumber
\end{align}
This goes to show that the function $f_{2}$ is strictly decreasing on $0 \leqslant x < \pi^{\frac{1}{d}} (1-p-q)^{\frac{1}{d}} d (d+1)^{-\frac{d-1}{d}}$ and strictly increasing on $x > \pi^{\frac{1}{d}} (1-p-q)^{\frac{1}{d}} d (d+1)^{-\frac{d-1}{d}}$, and its minimum is thus attained at $x=\pi^{\frac{1}{d}} (1-p-q)^{\frac{1}{d}} d (d+1)^{-\frac{d-1}{d}}$. This minimum value is given by 
\begin{align} 
&\pi (1-p-q) d^{d} \left\{\pi^{\frac{1}{d}} (1-p-q)^{\frac{1}{d}} d (d+1)^{-\frac{d-1}{d}}\right\}^{-d} (d+1)^{-d} + \left\{\pi^{\frac{1}{d}} (1-p-q)^{\frac{1}{d}} d (d+1)^{-\frac{d-1}{d}}\right\}\nonumber\\& (d+1)^{-\frac{1}{d}} \pi^{-\frac{1}{d}} (1-p-q)^{-\frac{1}{d}} = \frac{1}{d+1} + \frac{d}{d+1} = 1,\nonumber
\end{align}
thus establishing that \eqref{objective_binomial} does indeed hold for all values of $(p,q) \in I$. 

The last step left, keeping to the argument outlined in \S\ref{sec:gen_tech}, is to establish \eqref{x_{p,q}<alpha_{p,q}}, i.e.\ to show that $x_{p,q} < \alpha_{p,q}$ whenever \eqref{max_derivative_bounded_by_1_binomial} does not hold, and we prove this by showing that \eqref{g_{p,q}(x_{p,q})>x_{p,q}} holds, i.e.\ that $g_{p,q}(x_{p,q}) > x_{p,q}$ whenever \eqref{max_derivative_bounded_by_1_binomial} does not hold. From \eqref{g_{p,q}(x_{p,q})_binomial} and \eqref{x_{p,q,bin}}, the inequality $g_{p,q}(x_{p,q}) > x_{p,q}$ becomes equivalent to
\begin{align}
& 1 - \frac{1}{(d+1)\pi} - \frac{d q}{d+1} > \frac{1}{\pi}\left(\frac{1 - \pi q}{(d+1) \pi (1-p-q)}\right)^{\frac{1}{d}} + 1 - \frac{1}{\pi}\nonumber\\
\Longleftrightarrow & - 1 - d q \pi > (d+1)\left(\frac{1 - \pi q}{(d+1) \pi (1-p-q)}\right)^{\frac{1}{d}} - (d+1)\nonumber\\
\Longleftrightarrow & d (1-\pi q) > (1-\pi q)^{\frac{1}{d}} (d+1)^{\frac{d-1}{d}} \pi^{-\frac{1}{d}} (1-p-q)^{-\frac{1}{d}}\nonumber\\
\Longleftrightarrow & \pi^{\frac{1}{d}} (1-p-q)^{\frac{1}{d}} (1-\pi q)^{\frac{d-1}{d}} > d^{-1} (d+1)^{\frac{d-1}{d}},\nonumber
\end{align}
which, upon raising both sides to the $d$-th power, becomes precisely the negation of \eqref{max_derivative_bounded_by_1_binomial}. Therefore, we indeed have $g_{p,q}(x_{p,q}) > x_{p,q}$ whenever \eqref{max_derivative_bounded_by_1_binomial} fails, thus completing the last step of the argument outlined in \S\ref{sec:gen_tech} for the case where $\chi$ is Binomial$(d,\pi)$. This completes the proof of Theorem~\ref{thm:binomial_bond}.

\begin{remark}\label{rem:regular}
Setting $\pi = 1$ in the above set-up, we obtain a special sub-case of the Binomial$(d,\pi)$ distribution: where $\chi$ becomes the Dirac probability measure at $d$. In other words, $T_{\chi}$ is now replaced by the deterministic, rooted $d$-regular tree $T_{d}$, in which each vertex has precisely $d$ children. The above proof now reveals that the probability of draw in the bond percolation game played on $T_{d}$ is $0$ if and only if
\begin{equation}
(1-p-q) (1-q)^{d-1} \leqslant \frac{(d+1)^{d-1}}{d^{d}}.\nonumber
\end{equation}
\end{remark}

\subsection{The bond percolation game on the Poisson GW tree}\label{sec:Poisson_bond}
Let the offspring distribution $\chi$ of $T_{\chi}$ be Poisson$(\lambda)$, so that the probability generating function now becomes $G(x) = e^{\lambda(x-1)}$ (once again, the domain $\mathcal{D}$ on which this is defined, referring to the notation set in \S\ref{sec:gen_tech}, is the entire $\mathbb{R}^{+}$). Note that $G'(x) = \lambda G(x)$ and $G''(x) = \lambda^{2}G(x)$ in this case. The function $g_{p,q}$ now becomes
\begin{equation}
g_{p,q}(x) = (1-q) - (1-p-q)e^{\lambda(x-1)},\nonumber
\end{equation} 
and the function $g_{p,q}^{(2)}$ becomes
\begin{align}
g_{p,q}^{(2)}(x) &= (1-q) - (1-p-q)\exp\left\{\lambda\left[(1-q) - (1-p-q)e^{\lambda(x-1)} - 1\right]\right\}\nonumber\\
&= (1-q) - (1-p-q)\exp\left\{-\lambda q - \lambda (1-p-q)e^{\lambda(x-1)}\right\}.\nonumber
\end{align}
Differentiating the expression above, we obtain 
\begin{align}
\frac{d}{dx}g_{p,q}^{(2)}(x) &= (1-p-q)^{2} \lambda^{2}\exp\left\{\lambda(x-1)-\lambda q - \lambda (1-p-q) e^{\lambda(x-1)}\right\}.\label{g_{p,q,lambda}^{(2)}_derivative}
\end{align}
We note that the derivative 
\begin{align}
&\frac{d}{dx}\left[\lambda(x-1)-\lambda q - \lambda (1-p-q) e^{\lambda(x-1)}\right] = \lambda - \lambda^{2}(1-p-q)e^{\lambda(x-1)}\nonumber
\end{align} 
is strictly positive if and only if 
\begin{align}
e^{\lambda(x-1)} < \lambda^{-1}(1-p-q)^{-1} \Longleftrightarrow x < x_{p,q} \text{ where } x_{p,q} = 1 - \frac{\log \lambda}{\lambda} - \frac{\log(1-p-q)}{\lambda}.\label{x_{p,q,lambda}}
\end{align}
As argued in \S\ref{sec:gen_tech}, the maximum of the derivative of $g_{p,q}^{(2)}(x)$ is attained at $x=x_{p,q}$, and this maximum value equals, substituting $x=x_{p,q}$ from \eqref{x_{p,q,lambda}} in \eqref{g_{p,q,lambda}^{(2)}_derivative}, 
\begin{align}
&(1-p-q)^{2} \lambda^{2}\exp\left\{\lambda(x_{p,q}-1)-\lambda q - \lambda (1-p-q) e^{\lambda(x_{p,q}-1)}\right\} = (1-p-q) \lambda e^{-\lambda q - 1}.\label{max_g_{p,q,lambda}^{(2)}_derivative}
\end{align}
As argued in \S\ref{sec:gen_tech}, the probability of draw in the bond percolation game played on $T_{\chi}$, with $\chi$ Poisson$(\lambda)$, is $0$ whenever the inequality in \eqref{max_derivative_GW_leq_1} holds, i.e.\ whenever we have
\begin{equation}\label{max_g_{p,q,lambda}^{(2)}_derivative_bounded_by_1}
(1-p-q) \lambda e^{-\lambda q} \leqslant e.
\end{equation}

As the next step of the argument outlined in \S\ref{sec:gen_tech}, we check whether $w' \leqslant x_{p,q}$, which we accomplish by proving that \eqref{w'_leq_x_{p,q}} holds, i.e.\ by showing that $g_{p,q}^{(2)}(x_{p,q}) \leqslant x_{p,q}$ for all $(p,q) \in I$, with $I$ as defined in \eqref{I_defn}. In other words, we wish to show that
\begin{align}
&(1-q) - (1-p-q)e^{-\lambda q - 1} \leqslant 1 - \frac{\log \lambda}{\lambda} - \frac{\log(1-p-q)}{\lambda}\nonumber\\
&\Longleftrightarrow \log \lambda + \log(1-p-q) - \lambda q \leqslant (1-p-q)\lambda e^{-\lambda q - 1} \Longleftrightarrow \log x \leqslant x e^{-1},\label{objective_Poisson}  
\end{align}
where $x = (1-p-q)\lambda e^{-\lambda q - 1}$. Defining the function $f_{3}(x) = x e^{-1} - \log x$, we see that $f'_{3}(x) = e^{-1} - x^{-1}$, which is strictly positive as long as $x > e$. Thus, the function $f_{3}(x)$ is strictly decreasing for $0 < x < e$ and strictly increasing for $x > e$, leading to its minimum value being equal to $0$ (attained at $x=e$), and establishing that \eqref{objective_Poisson} does hold, as desired.

As the final step of the argument outlined in \S\ref{sec:gen_tech}, we show that \eqref{x_{p,q}<alpha_{p,q}} holds, i.e.\ that we have $x_{p,q} < \alpha_{p,q}$ whenever \eqref{max_g_{p,q,lambda}^{(2)}_derivative_bounded_by_1} is not true, and we accomplish this by showing that \eqref{g_{p,q}(x_{p,q})>x_{p,q}} is true, i.e.\ the inequality $g_{p,q}(x_{p,q}) > x_{p,q}$ holds whenever \eqref{max_g_{p,q,lambda}^{(2)}_derivative_bounded_by_1} does not, which is equivalent to
\begin{align}
&(1-q) - \frac{1}{\lambda} > 1 - \frac{\log \lambda}{\lambda} - \frac{\log(1-p-q)}{\lambda} \Longleftrightarrow \log \lambda + \log (1-p-q) - \lambda q > 1 \Longleftrightarrow (1-p-q)\lambda e^{-\lambda q} > e,\nonumber 
\end{align}
which is precisely the negation of \eqref{max_g_{p,q,lambda}^{(2)}_derivative_bounded_by_1}. This concludes the argument outlined in \S\ref{sec:gen_tech} for the case where $\chi$ is Poisson$(\lambda)$, and establishes the claim made in Theorem~\ref{thm:Poisson_bond}.

\subsection{The bond percolation game on a rooted GW tree with negative binomial offspring distribution}\label{sec:negative_binomial_bond}
We let $\chi$ be Negative Binomial$(r,\pi)$, i.e.\ where the number of children of the root $\phi$ of $T_{\chi}$ represents the total number of trials resulting in failures that are required in order to obtain precisely $r$ many successes, where $r \in \mathbb{N}$ is pre-specified. A trial here refers to that of a random experiment that either results in success with probability $\pi \in (0,1)$, or in failure with probability $1-\pi$, and the trials are assumed to be mutually independent. Thus, $\chi(m) = {m+r-1 \choose r-1} (1-\pi)^{m} \pi^{r}$ for all $m \in \mathbb{N}_{0}$, and the probability generating function becomes $G(x) = \pi^{r}\left\{1-(1-\pi)x\right\}^{-r}$ (here, the domain $\mathcal{D}$, mentioned in \S\ref{sec:gen_tech}, throughout which $G$ is well-defined is the interval $\left[0, (1-\pi)^{-1}\right)$). Therefore, the function $g_{p,q}$ now becomes
\begin{align}
g_{p,q}(x) &= (1-q) - (1-p-q)\left(\frac{\pi}{1-(1-\pi)x}\right)^{r}.\nonumber
\end{align}
We then have
\begin{align}
g'_{p,q}(x) &= \frac{-r(1-p-q)(1-\pi)\pi^{r}}{\left\{1-(1-\pi)x\right\}^{r+1}},\nonumber
\end{align}
leading to
\begin{align}
&\frac{d}{dx}g_{p,q}^{(2)}(x) = g'_{p,q}\left(g_{p,q}(x)\right) g'_{p,q}(x) = \frac{-r(1-p-q)(1-\pi)\pi^{r}}{\left\{1-(1-\pi)g_{p,q}(x)\right\}^{r+1}} \cdot \frac{-r(1-p-q)(1-\pi)\pi^{r}}{\left\{1-(1-\pi)x\right\}^{r+1}}\nonumber\\
&= \frac{r^{2}(1-p-q)^{2}(1-\pi)^{2}\pi^{2r}}{\left\{1 - (1-\pi)(1-q) + (1-\pi)(1-p-q)\left(\frac{\pi}{1-(1-\pi)x}\right)^{r}\right\}^{r+1} \left\{1-(1-\pi)x\right\}^{r+1}}\nonumber\\
&= \frac{r^{2}(1-p-q)^{2}(1-\pi)^{2}\pi^{2r}\left\{1-(1-\pi)x\right\}^{r^{2}-1}}{\left\{(q+\pi-q\pi)\left\{1-(1-\pi)x\right\}^{r} + (1-\pi)(1-p-q)\pi^{r}\right\}^{r+1}}\nonumber\\
&= r^{2}(1-p-q)^{2}(1-\pi)^{2}\pi^{2r} \left(\frac{\left\{1-(1-\pi)x\right\}^{r-1}}{(q+\pi-q\pi)\left\{1-(1-\pi)x\right\}^{r} + (1-\pi)(1-p-q)\pi^{r}}\right)^{r+1}\nonumber\\
&= r^{2}(1-p-q)^{2}(1-\pi)^{2}\pi^{2r} \left[(q+\pi-q\pi)\left\{1-(1-\pi)x\right\} + \frac{(1-\pi)(1-p-q)\pi^{r}}{\left\{1-(1-\pi)x\right\}^{(r-1)}}\right]^{-(r+1)}.\label{derivative_g_{p,q,pi,r}^{(2)}_negative_binomial}
\end{align}

From \eqref{derivative_g_{p,q,pi,r}^{(2)}_negative_binomial}, it is evident that if we define the function
\begin{align}
f_{4}(x) = (q+\pi-q\pi)\left\{1-(1-\pi)x\right\} + \frac{(1-\pi)(1-p-q)\pi^{r}}{\left\{1-(1-\pi)x\right\}^{(r-1)}},\nonumber
\end{align}
then the derivative of $g_{p,q}^{(2)}(x)$ is strictly increasing in $x$ whenever $f_{4}(x)$ is strictly decreasing in $x$, and vice versa. We see that
\begin{align}
f'_{4}(x) &= -(q+\pi-q\pi)(1-\pi) + \frac{(r-1)(1-\pi)^{2}(1-p-q)\pi^{r}}{\left\{1-(1-\pi)x\right\}^{r}},\nonumber
\end{align}
and this is strictly positive if and only if
\begin{align}
&\frac{(r-1)(1-\pi)^{2}(1-p-q)\pi^{r}}{\left\{1-(1-\pi)x\right\}^{r}} > (q+\pi-q\pi)(1-\pi)\nonumber\\
\Longleftrightarrow &\frac{(r-1)^{\frac{1}{r}}(1-\pi)^{\frac{1}{r}}(1-p-q)^{\frac{1}{r}}\pi}{(q+\pi-q\pi)^{\frac{1}{r}}} > 1-(1-\pi)x\nonumber\\
\Longleftrightarrow & x > x_{p,q}, \text{ where } x_{p,q} = \frac{1}{1-\pi}\left[1 - \frac{(r-1)^{\frac{1}{r}}(1-\pi)^{\frac{1}{r}}(1-p-q)^{\frac{1}{r}}\pi}{(q+\pi-q\pi)^{\frac{1}{r}}}\right].\label{f'_positive_cond}
\end{align}
As argued in \S\ref{sec:gen_tech}, the maximum of the derivative of $g_{p,q}^{(2)}(x)$ is attained at $x=x_{p,q}$, and using \eqref{derivative_g_{p,q,pi,r}^{(2)}_negative_binomial} and the expression for $x_{p,q}$ from \eqref{f'_positive_cond}, we see that this maximum value equals
\begin{align}
&r^{2}(1-p-q)^{2}(1-\pi)^{2}\pi^{2r} \left[(q+\pi-q\pi)\left\{1-(1-\pi)x_{p,q}\right\} + \frac{(1-\pi)(1-p-q)\pi^{r}}{\left\{1-(1-\pi)x_{p,q}\right\}^{(r-1)}}\right]^{-(r+1)} \nonumber\\
&= \frac{r^{2}(1-p-q)^{2}(1-\pi)^{2}\pi^{2r}\left\{1-(1-\pi)x_{p,q}\right\}^{(r-1)(r+1)}}{\left[(q+\pi-q\pi)\left\{1-(1-\pi)x_{p,q}\right\}^{r} + (1-\pi)(1-p-q)\pi^{r}\right]^{r+1}}\nonumber\\ 
&= \frac{r^{2}(1-p-q)^{2}(1-\pi)^{2}\pi^{2r}\left\{1-(1-\pi)x_{p,q}\right\}^{(r-1)(r+1)}}{\left[(q+\pi-q\pi) \cdot \frac{(r-1)(1-\pi)(1-p-q)\pi^{r}}{(q+\pi-q\pi)} + (1-\pi)(1-p-q)\pi^{r}\right]^{r+1}}\nonumber\\
&= \frac{r^{2}(1-p-q)^{2}(1-\pi)^{2}\pi^{2r}\left\{1-(1-\pi)x_{p,q}\right\}^{(r-1)(r+1)}}{r^{r+1} (1-\pi)^{r+1} (1-p-q)^{r+1} \pi^{r(r+1)}}\nonumber\\
&= \frac{\left\{1-(1-\pi)x_{p,q}\right\}^{(r-1)(r+1)}}{r^{r-1}(1-\pi)^{r-1}(1-p-q)^{r-1}\pi^{r(r-1)}} = \left(\frac{\left\{1-(1-\pi)x_{p,q}\right\}^{r+1}}{r (1-\pi) (1-p-q) \pi^{r}}\right)^{r-1}\nonumber\\
&= \left(\frac{\left\{1-(1-\pi) \cdot \frac{1}{1-\pi}\left[1 - \frac{(r-1)^{\frac{1}{r}}(1-\pi)^{\frac{1}{r}}(1-p-q)^{\frac{1}{r}}\pi}{(q+\pi-q\pi)^{\frac{1}{r}}}\right]\right\}^{r+1}}{r (1-\pi) (1-p-q) \pi^{r}}\right)^{r-1}\nonumber\\ 
&= \left(\frac{(r-1)^{\frac{r+1}{r}}(1-\pi)^{\frac{r+1}{r}}(1-p-q)^{\frac{r+1}{r}}\pi^{r+1}}{(q+\pi-q\pi)^{\frac{r+1}{r}}r (1-\pi) (1-p-q) \pi^{r}}\right)^{r-1} = \left(\frac{(r-1)^{\frac{r+1}{r}}(1-\pi)^{\frac{1}{r}}(1-p-q)^{\frac{1}{r}}\pi}{(q+\pi-q\pi)^{\frac{r+1}{r}}r}\right)^{r-1}.\nonumber
\end{align}
From the argument outlined in \S\ref{sec:gen_tech}, we conclude that the above expression is bounded above by $1$, and consequently, the probability of draw in the bond percolation game played on $T_{\chi}$, with $\chi$ Negative Binomial$(r,\pi)$, is $0$, whenever the inequality in \eqref{max_derivative_GW_leq_1} holds, which, in this case, is equivalent to
\begin{equation}
(r-1)^{r+1}(1-\pi)(1-p-q)\pi^{r} \leqslant (q+\pi-q\pi)^{r+1}r^{r}.\label{eq:lem_neg_bin_2}
\end{equation}

As the next step of the argument outlined in \S\ref{sec:gen_tech}, we show that $w' \leqslant x_{p,q}$, which we accomplish by showing that \eqref{w'_leq_x_{p,q}} is true, i.e.\ the inequality $g_{p,q}^{(2)}(x_{p,q}) \leqslant x_{p,q}$ holds for all $(p,q) \in I$, with $I$ as defined in \eqref{I_defn}. Note, using \eqref{f'_positive_cond}, that 
\begin{align}
&g_{p,q}(x_{p,q}) = (1-q) - (1-p-q)\left(\frac{\pi}{1-(1-\pi)x_{p,q}}\right)^{r}\nonumber\\
&= (1-q) - (1-p-q)\cdot \frac{q+\pi-q \pi}{(r-1) (1-\pi) (1-p-q)} = \frac{(1-q)(r-1)(1-\pi) - (q+\pi-q\pi)}{(r-1)(1-\pi)} \nonumber\\
&= \frac{(1-q)(r-1)(1-\pi) + (1-q)(1-\pi) - 1}{(1-\pi)(r-1)} = \frac{r(1-q)(1-\pi) - 1}{(1-\pi)(r-1)}.\label{g_{p,q,pi,r}(x_{p,q,pi,r})}
\end{align}
This, in turn, leads to 
\begin{align}
g_{p,q}^{(2)}(x_{p,q}) &= (1-q) - (1-p-q)\left(\frac{\pi}{1-(1-\pi) \cdot \frac{r(1-q)(1-\pi) - 1}{(1-\pi)(r-1)}}\right)^{r} = (1-q) - (1-p-q)\left(\frac{\pi (r-1)}{r(q+\pi-q\pi)}\right)^{r},\nonumber
\end{align}
so that we have $g_{p,q}^{(2)}(x_{p,q}) \leqslant x_{p,q}$ if and only if (using \eqref{f'_positive_cond})
\begin{align}
&(1-q) - (1-p-q)\left(\frac{\pi (r-1)}{r(q+\pi-q\pi)}\right)^{r} \leqslant \frac{1}{1-\pi}\left[1 - \frac{(r-1)^{\frac{1}{r}}(1-\pi)^{\frac{1}{r}}(1-p-q)^{\frac{1}{r}}\pi}{(q+\pi-q\pi)^{\frac{1}{r}}}\right]\nonumber\\
&\Longleftrightarrow \frac{(r-1)^{\frac{1}{r}}(1-\pi)^{\frac{1}{r}}(1-p-q)^{\frac{1}{r}}\pi}{(q+\pi-q\pi)^{\frac{1}{r}}} \leqslant (q+\pi-q\pi) + (1-\pi)(1-p-q)\left(\frac{\pi (r-1)}{r(q+\pi-q\pi)}\right)^{r}\nonumber\\
&\Longleftrightarrow \frac{(r-1)^{\frac{1}{r}}(1-\pi)^{\frac{1}{r}}(1-p-q)^{\frac{1}{r}}\pi}{(q+\pi-q\pi)^{\frac{r+1}{r}}} - \frac{(1-\pi)(1-p-q)\pi^{r}(r-1)^{r}}{r^{r}(q+\pi-q\pi)^{r+1}} \leqslant 1\nonumber\\
&\Longleftrightarrow (r-1)^{\frac{1}{r}}x - \frac{(r-1)^{r} x^{r}}{r^{r}} \leqslant 1, \text{ where } x = \frac{(1-\pi)^{\frac{1}{r}}(1-p-q)^{\frac{1}{r}}\pi}{(q+\pi-q\pi)^{\frac{r+1}{r}}}.\label{objective_neg_bin}
\end{align}
Let us define the function $f_{5}(x) = (r-1)^{\frac{1}{r}}x - (r-1)^{r} r^{-r} x^{r}$, so that the derivative
\begin{align}
f'_{5}(x) &= (r-1)^{\frac{1}{r}} - (r-1)^{r} r^{-r} \cdot r x^{r-1} = (r-1)^{\frac{1}{r}} - (r-1)^{r} r^{-(r-1)} x^{r-1}\nonumber
\end{align}
is strictly positive if and only if $(r-1)^{\frac{1}{r}} > (r-1)^{r} r^{-(r-1)} x^{r-1} \Longleftrightarrow x < r (r-1)^{-\frac{r+1}{r}}$. Therefore, the function $f_{5}(x)$ is strictly increasing on $0 \leqslant x < r (r-1)^{-\frac{r+1}{r}}$ and strictly decreasing on $r (r-1)^{-\frac{r+1}{r}} < x \leqslant 1$, and its maximum value is attained at $x=r (r-1)^{-\frac{r+1}{r}}$. This maximum value is given by
\begin{align}
&(r-1)^{\frac{1}{r}} \cdot r (r-1)^{-\frac{r+1}{r}} - (r-1)^{r} r^{-r} \cdot \left(r (r-1)^{-\frac{r+1}{r}}\right)^{r} = \frac{r}{r-1} - \frac{1}{r-1} = 1,\nonumber
\end{align}
which shows that \eqref{objective_neg_bin} does hold for all permissible values of $x$, thereby showing that \eqref{w'_leq_x_{p,q}} holds for all $(p,q) \in I$.

The final step of the argument outlined in \S\ref{sec:gen_tech} involves proving \eqref{x_{p,q}<alpha_{p,q}}, i.e.\ showing that $x_{p,q} < \alpha_{p,q}$ whenever \eqref{eq:lem_neg_bin_2} does not hold. This is accomplished by showing that \eqref{g_{p,q}(x_{p,q})>x_{p,q}} is true, i.e.\ the inequality $g_{p,q}(x_{p,q}) > x_{p,q}$ holds whenever \eqref{eq:lem_neg_bin_2} does not. Using \eqref{g_{p,q,pi,r}(x_{p,q,pi,r})} and \eqref{f'_positive_cond}, we see that this inequality is equivalent to 
\begin{align}
& \frac{r(1-q)(1-\pi) - 1}{(1-\pi)(r-1)} > \frac{1}{1-\pi}\left[1 - \frac{(r-1)^{\frac{1}{r}}(1-\pi)^{\frac{1}{r}}(1-p-q)^{\frac{1}{r}}\pi}{(q+\pi-q\pi)^{\frac{1}{r}}}\right]\nonumber\\
&\Longleftrightarrow \frac{r(1-q)(1-\pi)}{r-1} - \frac{1}{r-1} > 1 - \frac{(r-1)^{\frac{1}{r}}(1-\pi)^{\frac{1}{r}}(1-p-q)^{\frac{1}{r}}\pi}{(q+\pi-q\pi)^{\frac{1}{r}}}\nonumber\\
&\Longleftrightarrow \frac{(r-1)^{\frac{1}{r}}(1-\pi)^{\frac{1}{r}}(1-p-q)^{\frac{1}{r}}\pi}{(q+\pi-q\pi)^{\frac{1}{r}}} > 1 + \frac{1}{r-1} - \frac{r(1-q)(1-\pi)}{r-1}\nonumber\\
&\Longleftrightarrow \frac{(r-1)^{\frac{1}{r}}(1-\pi)^{\frac{1}{r}}(1-p-q)^{\frac{1}{r}}\pi}{(q+\pi-q\pi)^{\frac{1}{r}}} > \frac{r(q+\pi-q\pi)}{r-1}\nonumber\\
&\Longleftrightarrow (r-1)^{\frac{r+1}{r}}(1-\pi)^{\frac{1}{r}}(1-p-q)^{\frac{1}{r}}\pi > r(q+\pi-q\pi)^{\frac{r+1}{r}}\nonumber\\
&\Longleftrightarrow (r-1)^{r+1}(1-\pi)(1-p-q)\pi^{r} > r^{r}(q+\pi-q\pi)^{r+1},\nonumber
\end{align}
which is exactly what the negation of \eqref{eq:lem_neg_bin_2} guarantees. This completes the argument outlined in \S\ref{sec:gen_tech} in the case where $\chi$ is Negative Binomial$(r,\pi)$, and also concludes the proof of Theorem~\ref{thm:neg_bin_bond}.

\begin{remark}
It is worthwhile to note that when we set $r=1$ in \eqref{eq:lem_neg_bin_2}, the inequality is automatically satisfied for \emph{every} $(p,q) \in I$. This goes to show that when $\chi$ is the Geometric$(\pi)$ distribution, the probability of draw in the bond percolation game played on $T_{\chi}$ is $0$ for \emph{every} $(p,q) \in I$, where $I$ is as defined in \eqref{I_defn}.
\end{remark}

\subsection{When the offspring distribution is supported on $\{0,d\}$, for $d \in \mathbb{N}$ and $d \geqslant 2$}\label{sec:0_d_bond}
We consider $\chi$ such that $\chi(0) = 1-\pi$ and $\chi(d) = \pi$, for some $0 < \pi < 1$ and some $d \in \mathbb{N}$ with $d \geqslant 2$. In this case, the probability generating function is given by $G(x) = (1-\pi) + \pi x^{d}$, so that the function $g_{p,q}$ becomes 
\begin{align}
g_{p,q}(x) = (1-q) - (1-p-q)\left[(1-\pi) + \pi x^{d}\right] = \pi(1-q) + p(1-\pi) - (1-p-q)\pi x^{d}.\nonumber
\end{align}
This yields $g'_{p,q}(x) = -(1-p-q)\pi d x^{d-1}$, so that 
\begin{align}
\frac{d}{dx}g_{p,q}^{(2)}(x) &= (1-p-q)^{2}\pi^{2}d^{2}\left[\pi(1-q) + p(1-\pi) - (1-p-q)\pi x^{d}\right]^{d-1}x^{d-1}.\nonumber
\end{align}
Setting $f_{6}(x) = \left\{\pi(1-q) + p(1-\pi)\right\}x - (1-p-q)\pi x^{d+1}$, we see that $f'_{6}(x) = \pi(1-q) + p(1-\pi) - (1-p-q) \pi (d+1) x^{d}$, which is strictly positive if and only if 
\begin{align}
\pi(1-q) + p(1-\pi) > (1-p-q) \pi (d+1) x^{d} \Longleftrightarrow x < x_{p,q}, \text{ where } x_{p,q} = \left(\frac{\pi(1-q) + p(1-\pi)}{(1-p-q) \pi (d+1)}\right)^{\frac{1}{d}}.\label{x_{p,q,pi,d}}
\end{align}
As argued in \S\ref{sec:gen_tech}, the derivative of $g_{p,q}^{(2)}(x)$ attains its maximum value at $x=x_{p,q}$, and this maximum value equals
\begin{align}
&(1-p-q)^{2}\pi^{2}d^{2}\left[\pi(1-q) + p(1-\pi) - (1-p-q)\pi x_{p,q}^{d}\right]^{d-1}x_{p,q}^{d-1}\nonumber\\
&= (1-p-q)^{2}\pi^{2}d^{2}\left[\pi(1-q) + p(1-\pi) - (1-p-q)\pi \cdot \frac{\pi(1-q) + p(1-\pi)}{(1-p-q) \pi (d+1)}\right]^{d-1} \left(\frac{\pi(1-q) + p(1-\pi)}{(1-p-q) \pi (d+1)}\right)^{\frac{d-1}{d}}\nonumber\\
&= (1-p-q)^{2}\pi^{2}d^{2} \frac{\left[d\left\{\pi(1-q) + p(1-\pi)\right\}\right]^{d-1}}{(d+1)^{d-1}} \left(\frac{\pi(1-q) + p(1-\pi)}{(1-p-q) \pi (d+1)}\right)^{\frac{d-1}{d}}\nonumber\\
&= (1-p-q)^{\frac{d+1}{d}} \pi^{\frac{d+1}{d}} d^{d+1} \left\{\pi(1-q) + p(1-\pi)\right\}^{\frac{(d-1)(d+1)}{d}} (d+1)^{-\frac{(d-1)(d+1)}{d}}\nonumber\\
&= \left[(1-p-q)^{\frac{1}{d}} \pi^{\frac{1}{d}} d \left\{\pi(1-q) + p(1-\pi)\right\}^{\frac{d-1}{d}} (d+1)^{-\frac{d-1}{d}}\right]^{d+1},\nonumber
\end{align}
and this is bounded above by $1$ if and only if 
\begin{equation}\label{eq:lem_0_d_2}
(1-p-q) \pi \left\{\pi(1-q) + p(1-\pi)\right\}^{d-1} \leqslant \frac{(d+1)^{d-1}}{d^{d}}.
\end{equation}
This corresponds to \eqref{max_derivative_GW_leq_1} of \S\ref{sec:gen_tech}. As argued in \S\ref{sec:gen_tech}, the probability of the bond percolation game, played on $T_{\chi}$ with $\chi$ supported on $\{0,d\}$, is $0$ whenever \eqref{eq:lem_0_d_2} holds.

Following the argument outlined in \S\ref{sec:gen_tech}, our next task is to prove that $w' \leqslant x_{p,q}$, which we accomplish by proving \eqref{w'_leq_x_{p,q}}, i.e.\ by showing that $g_{p,q}^{(2)}(x_{p,q}) \leqslant x_{p,q}$ for each $(p,q) \in I$, where $I$ is as defined in \eqref{I_defn}. This is equivalent to showing that
\begin{align}
& \pi(1-q) + p(1-\pi) - (1-p-q)\pi \left[\pi(1-q) + p(1-\pi) - (1-p-q)\pi x_{p,q}^{d}\right]^{d} < x_{p,q}\nonumber\\
\Longleftrightarrow &\pi(1-q) + p(1-\pi) - (1-p-q)\pi \left[\pi(1-q) + p(1-\pi) - (1-p-q)\pi \cdot \frac{\pi(1-q) + p(1-\pi)}{(1-p-q) \pi (d+1)}\right]^{d} \nonumber\\&< \left(\frac{\pi(1-q) + p(1-\pi)}{(1-p-q) \pi (d+1)}\right)^{\frac{1}{d}}\nonumber\\
\Longleftrightarrow &\pi(1-q) + p(1-\pi) - \frac{(1-p-q)\pi d^{d} \left\{\pi(1-q) + p(1-\pi)\right\}^{d}}{(d+1)^{d}} < \left(\frac{\pi(1-q) + p(1-\pi)}{(1-p-q) \pi (d+1)}\right)^{\frac{1}{d}}\nonumber\\
\Longleftrightarrow & \left\{\pi(1-q) + p(1-\pi)\right\}\left[1 - \frac{(1-p-q)\pi d^{d} \left\{\pi(1-q) + p(1-\pi)\right\}^{d-1}}{(d+1)^{d}}\right] < \left(\frac{\pi(1-q) + p(1-\pi)}{(1-p-q) \pi (d+1)}\right)^{\frac{1}{d}}\nonumber\\
\Longleftrightarrow & 1 < \frac{x^{-\frac{1}{d}}}{(1-p-q)^{\frac{1}{d}} \pi^{\frac{1}{d}} (d+1)^{\frac{1}{d}}} + \frac{(1-p-q)\pi d^{d} x}{(d+1)^{d}}, \text{ where } x = \left\{\pi(1-q) + p(1-\pi)\right\}^{d-1}. \label{objective_0_d}
\end{align}
Defining the function $f_{7}(x) = x^{-\frac{1}{d}}(1-p-q)^{-\frac{1}{d}} \pi^{-\frac{1}{d}} (d+1)^{-\frac{1}{d}} + (1-p-q)\pi d^{d} x (d+1)^{-d}$, we obtain
\begin{align}
f'_{7}(x) = -\frac{1}{d} x^{-\frac{d+1}{d}} (1-p-q)^{-\frac{1}{d}} \pi^{-\frac{1}{d}} (d+1)^{-\frac{1}{d}} + (1-p-q)\pi d^{d} (d+1)^{-d},\nonumber
\end{align}
which is strictly positive if and only if
\begin{align}
&(1-p-q)\pi d^{d} (d+1)^{-d} > \frac{1}{d} x^{-\frac{d+1}{d}} (1-p-q)^{-\frac{1}{d}} \pi^{-\frac{1}{d}} (d+1)^{-\frac{1}{d}}\nonumber\\
\Longleftrightarrow & (1-p-q)^{\frac{d+1}{d}} \pi^{\frac{d+1}{d}} d^{d+1} (d+1)^{-\frac{(d-1)(d+1)}{d}} > x^{-\frac{d+1}{d}} \Longleftrightarrow x > (1-p-q)^{-1} \pi^{-1} d^{-d} (d+1)^{d-1}.\nonumber
\end{align}
Therefore, the function $f_{7}$ is strictly decreasing for $0 < x < (1-p-q)^{-1} \pi^{-1} d^{-d} (d+1)^{d-1}$, and strictly increasing for $(1-p-q)^{-1} \pi^{-1} d^{-d} (d+1)^{d-1} < x \leqslant 1$, so that the minimum value of $f_{7}(x)$ is attained at $x=(1-p-q)^{-1} \pi^{-1} d^{-d} (d+1)^{d-1}$. This minimum value is
\begin{align}
&\frac{\left\{(1-p-q)^{-1} \pi^{-1} d^{-d} (d+1)^{d-1}\right\}^{-\frac{1}{d}}}{(1-p-q)^{\frac{1}{d}} \pi^{\frac{1}{d}} (d+1)^{\frac{1}{d}}} + \frac{(1-p-q)\pi d^{d} (1-p-q)^{-1} \pi^{-1} d^{-d} (d+1)^{d-1}}{(d+1)^{d}} = \frac{d}{d+1} + \frac{1}{d+1} = 1,\nonumber
\end{align}
which tells us that \eqref{objective_0_d} holds for all $(p,q) \in I$, for $I$ as defined in \eqref{I_defn}.

The final step in the argument outlined in \S\ref{sec:gen_tech} is proving \eqref{x_{p,q}<alpha_{p,q}}, i.e.\ showing that $x_{p,q} < \alpha_{p,q}$ whenever \eqref{eq:lem_0_d_2} does not hold, which is accomplished by proving \eqref{g_{p,q}(x_{p,q})>x_{p,q}}, i.e.\ by showing that $g_{p,q}(x_{p,q}) > x_{p,q}$ holds whenever \eqref{eq:lem_0_d_2} does not. This inequality is equivalent to
(using \eqref{x_{p,q,pi,d}})
\begin{align}
&\pi(1-q) + p(1-\pi) - (1-p-q)\pi x_{p,q}^{d} > x_{p,q}\nonumber\\
\Longleftrightarrow & \pi(1-q) + p(1-\pi) - (1-p-q)\pi \cdot \frac{\pi(1-q) + p(1-\pi)}{(1-p-q) \pi (d+1)} > \left(\frac{\pi(1-q) + p(1-\pi)}{(1-p-q) \pi (d+1)}\right)^{\frac{1}{d}}\nonumber\\
\Longleftrightarrow & \frac{d\left\{\pi(1-q) + p(1-\pi)\right\}}{d+1} > \left(\frac{\pi(1-q) + p(1-\pi)}{(1-p-q) \pi (d+1)}\right)^{\frac{1}{d}}\nonumber\\
\Longleftrightarrow & \left\{\pi(1-q) + p(1-\pi)\right\}^{\frac{d-1}{d}} (1-p-q)^{\frac{1}{d}} \pi^{\frac{1}{d}} > (d+1)^{\frac{d-1}{d}} d^{-1}\nonumber\\
\Longleftrightarrow & \left\{\pi(1-q) + p(1-\pi)\right\}^{d-1} (1-p-q) \pi > \frac{(d+1)^{d-1}}{d^{d}},\nonumber
\end{align}
which is precisely what the negation of \eqref{eq:lem_0_d_2} is. This completes the argument outlined in \S\ref{sec:gen_tech} for the case where $\chi$ is supported on $\{0,d\}$, and establishes Theorem~\ref{thm:0_d_bond}.

\section{Ergodicity of a suitably defined probabilistic tree automaton}\label{sec:ergodicity}
This section is dedicated to exploring the connection between the bond percolation game played on $T_{d}$ (recall that $T_{d}$ is the rooted $d$-regular tree in which each vertex has precisely $d$ children, and it is obtained as a special case of $T_{\chi}$ in \S\ref{sec:binomial_bond} by setting $\pi = 1$ in the Binomial$(d,\pi)$ offspring distribution) and a probabilistic tree automaton (PTA) whose stochastic update rules represent the recurrence relations arising out of the game. We use the already established Theorem~\ref{thm:binomial_bond} to describe the regime in which this PTA is ergodic. Before proceeding further, we include here a brief introduction to what PTAs (defined on $T_{d}$) are, and what it means to say that such a PTA is ergodic.

\subsection{A formal introduction to probabilistic tree automata}\label{subsec:tree_automata}
A \emph{finite state probabilistic tree automaton} (PTA) $A$, defined on $T_{d}$, comprises the following components:
\begin{enumerate}
\item a \emph{finite} set $\mathcal{S}$ of \emph{states} or \emph{colours} that we refer to as its \emph{alphabet},
\item a \emph{stochastic update rule}, encompassed by a $|\mathcal{S}|^{d} \times |\mathcal{S}|$-dimensional stochastic matrix $\vartheta_{A}$, so that if a vertex $u$ in $T_{d}$ has children named $u_{1}, u_{2}, \ldots, u_{d}$, and $u_{i}$ has state $a_{i} \in \mathcal{S}$ for $1 \leqslant i \leqslant d$, the probability that the state of $u$ is $b \in \mathcal{S}$ is given by the entry $\vartheta_{A}((a_{1}, a_{2}, \ldots, a_{d}), b)$ of $\vartheta_{A}$.
\end{enumerate}
Let $L_{n}$ denote the set of all vertices of $T_{d}$ that are in generation $n$ for $n \in \mathbb{N}_{0}$, where the root $\phi$ is assumed to be in generation $0$ (i.e.\ $L_{0} = \{\phi\}$). Conditioned on the states of all vertices in generation $L_{n+1}$, the (random) states assigned to the vertices in $L_{n}$ under the application of $A$ are independent.

A \emph{configuration} on $L_{n}$ refers to an assignment $\sigma: L_{n} \rightarrow \mathcal{S}$ of states to the vertices of $L_{n}$, for each $n \in \mathbb{N}_{0}$. It is evident that once we fix a configuration on $L_{n}$, iterative applications of $A$ are only capable of determining / affecting the (random) states of the vertices in $L_{i}$ for $0 \leqslant i \leqslant n-1$. In fact, it will frequently be the case that we shall fix $n$, focus on the \emph{level-$n$ truncated tree} $T_{d}^{n}$ that is the induced subgraph of $T_{d}$ on the subset of vertices spanning generations $L_{0}, L_{1}, \ldots, L_{n}$, carry out our analysis and \emph{then} let $n \rightarrow \infty$. In such situations, a configuration on $L_{n}$, which is the last generation to be included in $T_{d}^{n}$, will be, quite appropriately, referred to as a \emph{boundary configuration}.

Given a boundary configuration $\sigma: L_{n} \rightarrow \mathcal{S}$ and a vertex $u$ that is situated in generation $i$ for some $0 \leqslant i \leqslant n-1$, we denote by $A(\sigma, u)$ the random variable indicating the state of $u$ obtained via iterative applications of $A$, starting from $\sigma$. Formally speaking, if $u$ is in $L_{n-1}$ and its children are named $u_{1}, u_{2}, \ldots, u_{d}$, then $A(\sigma, u)$ is the random variable with the following probability distribution:  
\begin{equation}
\Prob[A(\sigma, u) = b] = \vartheta_{A}((\sigma(u_{1}), \sigma(u_{2}), \ldots, \sigma(u_{d})), b), \text{ for each } b \in \mathcal{S},\nonumber
\end{equation}
where $\sigma(u_{i})$ indicates the state that $\sigma$ assigns to the vertex $u_{i}$ (which is in $L_{n}$) for each $1 \leqslant i \leqslant d$. For any $0 \leqslant i \leqslant n-2$, having defined $A(\sigma, w)$ for every vertex $w$ that lies in $L_{j}$ for all $i < j \leqslant n-1$, we define $A(\sigma, v)$, for every $v \in L_{i}$, as follows. Suppose $v$ has children $v_{1}, v_{2}, \ldots, v_{d}$. Then
\begin{align}
\Prob\left[A(\sigma, v) = b\big|A(\sigma, v_{1}) = a_{1}, A(\sigma, v_{2}) = a_{2}, \ldots, A(\sigma, v_{d}) = a_{d}\right] = \vartheta_{A}((a_{1}, a_{2}, \ldots, a_{d}), b),
\end{align}
for all $b, a_{1}, a_{2}, \ldots, a_{d} \in \mathcal{S}$. Given a (deterministic) boundary configuration $\sigma$ on $L_{n}$ for any $n \in \mathbb{N}$, and $u$ a vertex in $L_{i}$ for any $0 \leqslant i \leqslant n-1$, we let $\mu_{A}(\sigma, u)$ denote the law of the random variable $A(\sigma, u)$. When $u$ equals the root $\phi$ of $T_{d}$, we abbreviate $\mu_{A}(\sigma, \phi)$ as simply $\mu_{A}(\sigma)$.

The notations introduced in the previous paragraph can be extended as follows when, instead of considering a deterministic boundary configuration $\sigma$ on $L_{n}$, we consider a \emph{random} boundary configuration $\pmb{\sigma}$ on $L_{n}$, with law $\nu$ supported on $\mathcal{S}^{d^{n}}$. For any $0 \leqslant i \leqslant n-1$ and any vertex $u$ that lies in $L_{i}$, we let $A(\pmb{\sigma}, u)$ denote the (random) state of $u$ obtained by starting from $\pmb{\sigma}$ on $L_{n}$ and iteratively applying $A$ up the generations of $T_{d}^{n}$. We let $\mu_{A}(\nu, u)$ denote the law of $A(\pmb{\sigma}, u)$. As above, we abbreviate $\mu_{A}(\nu, \phi)$ as $\mu_{A}(\nu)$. 

Recall that, given a measurable space $(\Omega, \mathcal{F})$ and two probability distributions $\mu_{1}$ and $\mu_{2}$ defined on it, the \emph{total variation distance} between $\mu_{1}$ and $\mu_{2}$ is defined as
\begin{equation}
\left|\left|\mu_{1} - \mu_{2}\right|\right|_{\tv} = \sup_{S \in \mathcal{F}}\left|\mu_{1}(S) - \mu_{2}(S)\right|.\nonumber
\end{equation}
When $\Omega$ is countable and $\mathcal{F}$ is its power set, the total variation metric is related to the $L^{1}$ metric via the following relation:
\begin{equation}
\left|\left|\mu_{1} - \mu_{2}\right|\right|_{\tv} = \frac{1}{2}\left|\left|\mu_{1} - \mu_{2}\right|\right|_{L^{1}} = \frac{1}{2}\sum_{\omega \in \Omega}\left|\mu_{1}(\{\omega\}) - \mu_{2}(\{\omega\})\right|.\nonumber
\end{equation}
Moreover, we have (see, for instance, Proposition 4.7 of \cite{levin2017markov})
\begin{equation}
\left|\left|\mu_{1} - \mu_{2}\right|\right|_{\tv} = \inf\left\{\Prob[X_{1} \neq X_{2}]: (X_{1}, X_{2}) \text{ is a coupling of } \mu_{1} \text{ and } \mu_{2}\right\},\label{tv_coupling_relation}
\end{equation}
where by a coupling $(X_{1}, X_{2})$ of $\mu_{1}$ and $\mu_{2}$ we mean a pair of random variables, $X_{1}$ and $X_{2}$, defined on the same probability space (with a specified joint distribution), such that $X_{1}$ follows the law $\mu_{1}$ and $X_{2}$ follows the law $\mu_{2}$.

\begin{defn}\label{defn:ergodicity}
We call a PTA $A$ \emph{ergodic} if 
\begin{equation}
\limsup_{n \rightarrow \infty}\sup_{\sigma, \tau: L_{n} \rightarrow \mathcal{S}}\left|\left|\mu_{A}(\sigma) - \mu_{A}(\tau)\right|\right|_{\tv} = 0,\label{corr_decay}
\end{equation}
where $\mu_{A}(\sigma)$ and $\mu_{A}(\tau)$, as defined above, indicate the laws of the random variables $A(\sigma,\phi)$ and $A(\tau,\phi)$ respectively.
\end{defn}
A definition this instrumental deserves some discussions about what truly its implications are. Consider \emph{any} two boundary configurations $\sigma$ and $\tau$ on $L_{n}$. What the probability distributions $\mu_{A}(\sigma)$ and $\mu_{A}(\tau)$ of the random variables $A(\sigma, \phi)$ and $A(\tau, \phi)$ capture are the ``effects" of the boundary configurations $\sigma$ and $\tau$, respectively, on the state of the root $\phi$ of $T_{d}$, under iterative applications of the PTA $A$. Consequently, what the total variation distance $\left|\left|\mu_{A}(\sigma) - \mu_{A}(\tau)\right|\right|_{\tv}$ captures is the difference between these effects, i.e.\ how differently the probability distribution of the state of $\phi$ is impacted because of starting from the boundary configuration $\sigma$ on $L_{n}$ as opposed to starting from the boundary configuration $\tau$ on $L_{n}$. When \eqref{corr_decay} holds, it is an indication that the effect \emph{any} boundary configuration assigned to $L_{n}$ has on the state of $\phi$ \emph{dwindles} or \emph{decays} as $n$ approaches $\infty$. It is as if, when $n$ is sufficiently large, the PTA $A$, viewed as a stochastic operator that acts on the states of the vertices of $L_{i}$ for all $1 \leqslant i \leqslant n$, tends to nearly \emph{forget} the ``initial" boundary configuration on $L_{n}$.

\subsection{The PTA defined via the game's recurrence relations}\label{subsec:game_PTA}
Recall the sets $W$, $L$ and $D$ introduced in \S\ref{sec:bond_GW}, and the recurrence relations described in \S\ref{subsec:recurrence_bond}. It is evident, applying the same argument as adopted in \S\ref{subsec:recurrence_bond}, that for any vertex $v \in T_{d}$ with children named $v_{1}, \ldots, v_{d}$, we have $v \in W$ if and only if either the edge $(v,v_{k})$ is a target for some $1 \leqslant k \leqslant d$, or there exists some $1 \leqslant \ell \leqslant d$ such that $v_{\ell} \in L$ and the edge $(v,v_{\ell})$ is safe. Likewise, $v \in D$ if and only if no $(v,v_{\ell})$ is a target, $(v,v_{\ell})$ is a trap for every $v_{\ell}$ that belongs to $L$, and there exists at least one $v_{k} \in D$ with $(v,v_{k})$ safe. In all other cases, $v \in L$. 

The recurrence relations mentioned above can be summarized in the form of a PTA $B_{p,q}$, as follows. The associated alphabet is $\mathcal{S} = \{W, L, D\}$, and the stochastic update rule is captured by the $3^{d} \times 3$-dimensional stochastic matrix $\vartheta_{B_{p,q}}$ whose elements are given by
\begin{equation}\label{B_{p}}
\vartheta_{B_{p,q}}((a_{1}, a_{2}, \ldots, a_{d}), b) = 
  \begin{cases} 
   1 - p^{i}(1-q)^{d-i} & \text{when } b = W, \\
   p^{i}(1-q)^{d-i} - p^{i+j}(1-q)^{d-i-j} & \text{when } b = D,\\
   p^{i+j}(1-q)^{d-i-j} & \text{when } b = L,
  \end{cases}
\end{equation}
for all $a_{1}, a_{2}, \ldots, a_{d} \in \mathcal{S}$ such that $i = \left|\left\{1 \leqslant \ell \leqslant d: a_{\ell} = L\right\}\right|$ and $j = \left|\left\{1 \leqslant \ell \leqslant d: a_{\ell} = D\right\}\right|$.

\subsection{How the probability of draw ties in with the ergodicity of this PTA}\label{subsec:draw_ergodicity}
\begin{theorem}\label{thm:draw_ergodicity_bond}
The probability of the event that the bond percolation game with parameter-pair $(p,q)$ results in a draw is $0$ if and only if $B_{p,q}$ is ergodic.
\end{theorem}
\begin{proof}
We let $\pmb{\omega}$ denote the random assignment of trap / target / safe labels to the edges of $T_{d}$. Given a vertex $u$ of $T_{d}$, we let $\pmb{\omega}(u)$ denote the restriction of $\pmb{\omega}$ to the edges of $T_{d}(u)$, where $T_{d}(u)$ denotes the sub-tree of $T_{d}$ induced on the set of vertices comprising $u$ and all its descendants. Let $X(u)$ denote the (random) state (in $\mathcal{S} = \{W,L,D\}$) of $u$ obtained as a result of the bond percolation game played using the random assignment $\pmb{\omega}$ and with the token placed at $u$ at the start of the game. We abbreviate $X(\phi)$ by $X$, and let $\mu$ denote the law of $X$. Note, crucially, that $X(u)$ is actually a function of $\pmb{\omega}(u)$, and its law is also $\mu$.

From above, it is evident that if we assign i.i.d.\ $\mu$ random states to the vertices of $L_{n}$, then the law of the random state of $\phi$ obtained via iterative applications of $B_{p,q}$ is again $\mu$. Formally, letting $\nu^{(n)}$ denote the probability distribution
\begin{equation}
\nu^{(n)}\left[\left(a_{1}, a_{2}, \ldots, a_{d^{n}}\right)\right] = \prod_{i=1}^{d^{n}}\mu[a_{i}], \text{ for all } a_{1}, a_{2}, \ldots, a_{d^{n}} \in \mathcal{S},\nonumber
\end{equation}
we conclude that the law $\mu_{B_{p,q}}(\nu^{(n)})$ is the same as $\mu$. 

Let us denote by $\delta_{L}^{n}$ the boundary configuration on $L_{n}$ that assigns the state $L$ to each vertex of $L_{n}$. From \eqref{B_{p}}, it is evident that $\mu_{B_{p,q}}\left(\delta_{L}^{n}\right)[D] = 0$ for every $n \in \mathbb{N}$ (in fact, this assertion is true for \emph{any} boundary configuration on $L_{n}$ that assigns the state $D$ to \emph{none} of the vertices in $L_{n}$). This fact, along with \eqref{corr_decay}, implies that when $B_{p,q}$ is ergodic,
\begin{align}
\limsup_{n \rightarrow \infty}\mu_{B_{p,q}}(\nu^{(n)})[D] = \limsup_{n \rightarrow \infty}\left|\mu_{B_{p,q}}(\nu^{(n)})[D] - \mu_{B_{p,q}}\left(\delta_{L}^{n}\right)[D]\right| \leqslant \limsup_{n \rightarrow \infty}\left|\left|\mu_{B_{p,q}}(\nu^{(n)}) - \mu_{B_{p,q}}\left(\delta_{L}^{n}\right)\right|\right|_{\tv} = 0,\nonumber
\end{align}
so that $\mu[D] = \lim_{n \rightarrow \infty}\mu_{B_{p,q}}(\nu^{(n)})[D] = 0$ when $B_{p,q}$ is ergodic. This shows that when $B_{p,q}$ is ergodic, the probability that the root $\phi$ has state $D$ under $\pmb{\omega}$ is $0$, so that the bond percolation game starting with $\phi$ as the initial vertex \emph{almost never} results in a draw.

We now assume that the bond percolation game beginning at $\phi$ has probability $0$ of resulting in a draw. Let $\Omega$ denote the set of all realizations of $\pmb{\omega}$ for which the game starting at $\phi$ does not lead to a draw. By our hypothesis, the set $\{\text{trap}, \text{target}, \text{safe}\}^{E(T_{d})} \setminus \Omega$, where $E(T_{d})$ denotes the set of all edges of $T_{d}$, is of measure $0$. Furthermore, our hypothesis guarantees the existence of a random variable $N$, taking values in $\mathbb{N}$ and \emph{finite almost surely}, such that the game terminates in less than $N$ rounds. In fact, for each realization $\omega$ of $\pmb{\omega}$ that belongs to $\Omega$, the corresponding value $N(\omega)$ of $N$ is finite. 

Let $u$ be a vertex in $T_{d}$, with children $u_{1}, \ldots, u_{d}$, and let $\omega(u,u_{i})$, for $1 \leqslant i \leqslant d$, denote the deterministic trap / target / safe labels assigned to the edges $(u,u_{i})$, for $1 \leqslant i \leqslant d$, under $\omega$. We also fix a deterministic assignment of states $\sigma: \{u_{1}, \ldots, u_{d}\} \rightarrow \mathcal{S}$. The state of $u$ is then decided as follows:
\begin{enumerate}
\item \label{i} If there exists some $1 \leqslant i \leqslant d$ with $\omega(u,u_{i})$ safe and $\sigma(u_{i}) = L$, then $u$ is in $W$.
\item \label{ii} If the above does not happen, but there exists some $1 \leqslant j \leqslant d$ with $\omega(u,u_{j})$ safe and $\sigma(u_{j}) = D$, then $u$ is in $D$.
\item \label{iii} In all other cases, $u$ is in $L$.
\end{enumerate}

Fix \emph{any} $\omega$ in $\Omega$ and \emph{any} deterministic boundary configuration $\sigma: L_{N(\omega)} \rightarrow \mathcal{S}$ on $L_{N(\omega)}$. Starting from $\sigma$, we use $\omega$ and the rules described above to decide the state $B_{p,q}(\sigma,u)(\omega)$ (recall this notation from \S\ref{subsec:tree_automata}) of the root $\phi$ -- however, since the fate of the game that begins from $\phi$ is decided in \emph{less} than $N(\omega)$ rounds, $\sigma$ has no effect whatsoever on this state. Using the notation introduced at the beginning of this proof, we conclude that $B_{p,q}(\sigma,u)(\omega) = X(\omega)$ (where $X(\omega)$ is dictated by $\omega$ alone), irrespective of our choice of $\sigma$. Letting $S_{n}$, for every $n \in \mathbb{N}$, denote the set of all $\omega \in \Omega$ for which $N(\omega) \leqslant n$, we have
\begin{align}
&B_{p,q}(\sigma,\phi)(\omega) = X(\omega) \text{ for every boundary configuration } \sigma \text{ on } L_{n}, \text{ for each } \omega \in S_{n}\nonumber\\
&\implies B_{p,q}(\sigma,\phi)\mathbf{1}_{S_{n}} = X \mathbf{1}_{S_{n}} \text{ for every boundary configuration } \sigma \text{ on } L_{n}\nonumber\\
&\implies \Prob\left[B_{p,q}(\sigma,\phi) \neq X\right] \leqslant 1-\Prob[S_{n}] = \Prob[N > n], \text{ for every boundary configuration } \sigma \text{ on } L_{n}.
\end{align}
Consequently, for \emph{any} two boundary configurations $\sigma$ and $\tau$ on $L_{n}$, we have
\begin{align}
& \Prob\left[B_{p,q}(\sigma,\phi) \neq B_{p,q}(\tau,\phi)\right] \leqslant \Prob\left[B_{p,q}(\sigma,\phi) \neq X\right] + \Prob\left[B_{p,q}(\tau,\phi) \neq X\right] \leqslant 2\Prob[N > n].\label{sigma_tau_unequal_bound}
\end{align}
From \eqref{sigma_tau_unequal_bound} and the fact that $N$ is finite almost surely, we conclude that
\begin{align}
& \limsup_{n \rightarrow \infty}\sup_{\sigma, \tau: L_{n} \rightarrow \mathcal{S}}\Prob\left[B_{p,q}(\sigma,\phi) \neq B_{p,q}(\tau,\phi)\right] \leqslant \limsup_{n \rightarrow \infty} 2\Prob[N > n] = 0,\label{sup_sigma_tau_unequal_bound}
\end{align}
so that from \eqref{sup_sigma_tau_unequal_bound}, \eqref{tv_coupling_relation} and \eqref{corr_decay}, we conclude that $B_{p,q}$ is indeed ergodic, as desired.
\end{proof}

Theorem~\ref{thm:draw_ergodicity_bond}, along with Remark~\ref{rem:regular}, tells us that the PTA $B_{p,q}$ is ergodic if and only if $(1-p-q) (1-q)^{d-1} \leqslant (d+1)^{d-1} d^{-d}$.

\section{Condition that guarantees that the expected duration of the bond percolation game is finite}\label{sec:finite_expected_duration}
While the probability of draw being $0$ implies that the duration of the bond percolation game (played on $T_{\chi}$, for \emph{any} offspring distribution $\chi$), henceforth denoted by the random variable $\mathcal{T}$, is finite almost surely, it does not guarantee that its expected value, i.e.\ $\E[\mathcal{T}]$, is finite as well. Here, we provide a sufficient condition for $\E[\mathcal{T}]$ to be finite.

To begin with, we assume that we are in a regime where the probability of draw is $0$. From parts \eqref{main_GW_bond_1} and \eqref{main_GW_bond_3} of Theorem~\ref{thm:main_GW_bond}, this means that $w' = p + (1-p-q)w$ is the unique fixed point of the function $g_{p,q}^{(2)}$ in $[0,1]$. Recall, from \S\ref{sec:bond_GW}, the sets $W_{n}$, $L_{n}$ and $D_{n}$, the probabilities $w_{n}$ and $\ell_{n}$, the transformed variables $w'_{n} = p + (1-p-q)w_{n}$ and $\ell'_{n} = (1-p-q)\ell_{n}$, and the recurrence relations described in \S\ref{subsec:recurrence_bond}. We can write
\begin{align}
\E[\mathcal{T}] &= \sum_{n=0}^{\infty}\Prob[\mathcal{T} \geqslant n+1] = \sum_{n=0}^{\infty}\Prob[\phi \in D_{n}]\nonumber\\
&= \sum_{n=0}^{\infty}\left\{1 - \Prob[\phi \in W_{n}] - \Prob[\phi \in L_{n}]\right\} = \sum_{n=0}^{\infty}\left\{1 - w_{n} - \ell_{n}\right\}\nonumber\\
&= 1 + \frac{1}{1-p-q}\sum_{n=1}^{\infty}\left\{1-p-q - (1-p-q)w_{n} - (1-p-q)\ell_{n}\right\} \nonumber\\
&= 1 + \frac{1}{1-p-q}\sum_{n=1}^{\infty}\left\{1-q - w'_{n} - \ell'_{n}\right\} = 1 + \frac{1}{1-p-q}\sum_{n=1}^{\infty}\left\{w' + \ell' - w'_{n} - \ell'_{n}\right\}\nonumber\\
&= 1 + \frac{1}{1-p-q}\sum_{n=1}^{\infty}\left\{w'-w'_{n}\right\} + \frac{1}{1-p-q}\sum_{n=1}^{\infty}\left\{\ell'-\ell'_{n}\right\}.\label{avg_dur_0}
\end{align}
Using the recurrence relation \eqref{GW_recurrence_3}, and the mean value theorem, we have
\begin{align}
w'-w'_{n+2} &= g_{p,q}^{(2)}(w') - g_{p,q}^{(2)}\left(w'_{n}\right) = \frac{d}{dx}g_{p,q}^{(2)}(x)\Big|_{x=\xi_{n}}\left(w'-w'_{n}\right),\label{avg_dur_1}
\end{align}
where $\xi_{n}$ is some quantity with $w'_{n} < \xi_{n} < w'$. Likewise, using \eqref{GW_recurrence_2}, the definition of $g_{p,q}$ from \eqref{GW_g_defn}, and part \eqref{main_GW_bond_2} of Theorem~\ref{thm:main_GW_bond}, we have
\begin{align}
\ell'-\ell'_{n+2} &= g_{p,q}(w'_{n+1}) - g_{p,q}(w') = g'_{p,q}\left(\zeta_{n+1}\right)\left(w'_{n+1}-w'\right) = \left|g'_{p,q}\left(\zeta_{n+1}\right)\right|\left(w' - w'_{n+1}\right),\label{avg_dur_2}
\end{align}
for some quantity $\zeta_{n+1}$, with $w'_{n+1} < \zeta_{n+1} < w'$, whose existence is guaranteed by the mean value theorem. Note that, under the assumption that the probability of draw is $0$, we have $w' = \alpha_{p,q}$, where $\alpha_{p,q}$ is the unique fixed point of $g_{p,q}$ in $[0,1]$.

In \S\ref{subsec:recurrence_bond}, we showed that $w'$ is the smallest positive fixed point of $g_{p,q}^{(2)}$, and in \eqref{g_{p,q}^{(2)}(0)_lower_bound_1} and \eqref{g_{p,q}^{(2)}(0)_lower_bound_2}, we showed that the curve $y=g_{p,q}^{(2)}(x)$ lies above the line $y=x$ at $x=0$. Consequently, the curve $y=g_{p,q}^{(2)}(x)$ must travel from \emph{above} the line $y=x$ to \emph{beneath} the line $y=x$ at $x=w'$. Therefore, the slope of $y=g_{p,q}^{(2)}(x)$ must be bounded above by the slope of $y=x$, which is $1$, at $x=w'$. In other words, $\frac{d}{dx}g_{p,q}^{(2)}(x)\Big|_{x=w'} \leqslant 1$.

We now assume that
\begin{equation}\label{assumption_finite_expectation_duration}
\frac{d}{dx}g_{p,q}^{(2)}(x)\Big|_{x=w'} < 1 \text{ and } \left|\frac{d}{dx}g_{p,q}(x)\Big|_{x=w'}\right| < 1.
\end{equation}
Let $\frac{d}{dx}g_{p,q}^{(2)}(x)\Big|_{x=w'} = \alpha$ and $\left|\frac{d}{dx}g_{p,q}(x)\Big|_{x=w'}\right| = \beta$, so that under the assumption made in \eqref{assumption_finite_expectation_duration}, $\alpha < 1$ and $\beta < 1$. Let us choose and fix $0 < \epsilon < \min\{1-\alpha, 1-\beta\}$. From \eqref{w_{n}_ell_{n}_limits}, we know that $w'_{n} \uparrow w'$ as $n \rightarrow \infty$. By the continuity of both the functions $\frac{d}{dx}g_{p,q}^{(2)}(x)$ and $\left|\frac{d}{dx}g_{p,q}(x)\right|$ in $x$, we know that there exists some $N \in \mathbb{N}$ such that 
\begin{equation}
\frac{d}{dx}g_{p,q}^{(2)}(x) < \alpha+\epsilon \text{ and } \left|\frac{d}{dx}g_{p,q}(x)\right| < \beta+\epsilon \text{ for all } w'_{n} \leqslant x \leqslant w', \text{ for all } n \geqslant N.\nonumber
\end{equation} 
This, in turn, implies that
\begin{equation}\label{avg_dur_3}
\frac{d}{dx}g_{p,q}^{(2)}(x)\Big|_{x=\xi_{n}} < \alpha+\epsilon < 1 \text{ and } \left|g'_{p,q}\left(\zeta_{n+1}\right)\right| < \beta +\epsilon < 1 \text{ for all } n \geqslant N.
\end{equation}

From \eqref{avg_dur_1}, \eqref{avg_dur_2} and \eqref{avg_dur_3}, we see that the terms corresponding to $n \geqslant N$ in the infinite series of \eqref{avg_dur_0} decay exponentially fast, thereby allowing the series to converge. This completes the proof of the fact that under \eqref{assumption_finite_expectation_duration}, the expected duration of the bond percolation game is finite. This also completes the proof of Theorem~\ref{thm:avg_dur}.

\bibliography{Tree_automata_paper_bib}
\end{document}